\documentclass[reqno,twoside]{amsart}

\usepackage{amssymb,amsmath,enumerate,extarrows}
\usepackage[colorlinks=true,linkcolor=black,citecolor=black]{hyperref}
\usepackage{color}
\usepackage[utf8]{vietnam}

\usepackage[all]{xy}
\newcommand{\mm}{\mathfrak m}
\newcommand{\nn}{\mathfrak n}

\newcommand{\qq}{\mathfrak q}

\newcommand{\N}{\mathbb{N}}
\newcommand{\Q}{\mathbb{Q}}

\newcommand{\Fcc}{\mathcal{F}}

\DeclareMathOperator{\ann}{ann}
\DeclareMathOperator{\chara}{char}
\DeclareMathOperator{\Coker}{Coker}
\DeclareMathOperator{\depth}{depth}

\DeclareMathOperator{\gr}{gr}
\DeclareMathOperator{\Ker}{Ker}
\DeclareMathOperator{\lind}{ld}
\DeclareMathOperator{\linp}{lin}
\DeclareMathOperator{\pdim}{pd}
\DeclareMathOperator{\reg}{reg}
\DeclareMathOperator{\rstab}{rstab}
\DeclareMathOperator{\supp}{supp}
\DeclareMathOperator{\Tor}{Tor}

\newtheorem{thm}{\bf Theorem}[section]
\newtheorem{lem}[thm]{\bf Lemma}
\newtheorem{cor}[thm]{\bf Corollary}
\newtheorem{prop}[thm]{\bf Proposition}
\newtheorem{quest}[thm]{\bf Question}

\theoremstyle{definition}

\newtheorem{notn}[thm]{\bf Notation}

\theoremstyle{remark}
\newtheorem{rem}[thm]{Remark}

\theoremstyle{remark}
\newtheorem*{finrem*}{Final remark}

\newcommand{\vepsilon}{\varepsilon}

\numberwithin{equation}{section}

\title[Powers of fiber products]{Homological invariants of powers of fiber products}

\dedicatory{Dedicated to  Professor L\^e Tuấn Hoa on the occasion of his sixtieth birthday}

\author{Hop D. Nguyen}
\address{Institute of Mathematics, Vietnam Academy of Science and Technology, 18 Hoang Quoc Viet, 10307 Hanoi, Vietnam}
\email{ngdhop@gmail.com}
\author{Thanh Vu}
\address{Hanoi University of Science and Technology, 1 Dai Co Viet, Hai Ba Trung, Hanoi, Vietnam}
\email{vuqthanh@gmail.com}


\begin{document}
\renewcommand{\abstractname}{Abstract}
\renewcommand{\appendixname}{Appendix}
\renewcommand\refname{References}
\renewcommand\proofname{Proof}

\begin{abstract}
Let $R$ and $S$ be polynomial rings of positive dimensions over a field $k$. Let $I\subseteq R, J\subseteq S$ be non-zero homogeneous ideals none of which contains a linear form. Denote by $F$ the fiber product of $I$ and $J$ in $T=R\otimes_k S$. We compute homological invariants of the powers of $F$ using the data of $I$ and $J$. Under the assumption that either $\chara k=0$ or $I$ and $J$ are monomial ideals, we provide explicit formulas for the depth and regularity of powers of $F$. In particular, we establish for all $s\ge 2$ the intriguing formula $\depth(T/F^s)=0$. If moreover each of the ideals $I$ and $J$ is generated in a single degree, we show that for all $s\ge 1$, $\reg F^s=\max_{i\in [1,s]}\{\reg I^i+s-i,\reg J^i+s-i\}$. Finally, we prove that the linearity defect of $F$ is the maximum of the linearity defects of $I$ and $J$, extending previous work of Conca and R\"omer. The proofs exploit the so-called Betti splittings of powers of a fiber product.
\end{abstract}

\maketitle

\section{Introduction}
In this paper, we are concerned with the theory of powers of ideals. There is many work on the asymptotic properties of powers of ideals; see, e.g., \cite{Br,Ch,CHT,EH,EU,HH2,HPV,Kod}. The reader is referred to a recent survey of Chardin \cite{Ch2} for an overview of work on regularity of powers of ideals. The novel feature of the present work is its focus on {\it exact} formulas for homological invariants of {\it all} the powers, under some fairly general circumstances.

Let $R,S$ be standard graded polynomial rings over a field $k$. Denote by $\mm$ and $\nn$ the corresponding graded maximal ideals of $R$ and $S$, and $T=R\otimes_k S$. Let $I\subseteq \mm^2, J\subseteq \nn^2$ be homogeneous ideals. By abuse of notation, we also denote by $I$ and $J$ the extensions of these ideals to $T$. In \cite{HTT} and \cite{NgV}, formulas for the depth and regularity of powers of $I+J$ were provided. One motivation for both papers is that the sum $I+J$ defines a fundamental operation on the $k$-algebras $R/I$ and $S/J$, namely their tensor product. Among others, fiber product is also a fundamental operation on $k$-algebras. Indeed, fiber products of algebras (and more generally, of rings) were investigated by many authors; see, for example, \cite{AAF,CR,DK,Les,Mo,NSW}. Now the fiber product of the $k$-algebras $R/I$ and $S/J$ is defined by $F=I+J+\mm\nn$, which accordingly is called the {\it fiber product} of $I$ and $J$. To avoid triviality, we will assume in the rest of this introduction that $\mm,\nn\neq 0$, i.e.~ $R$ and $S$ have positive Krull dimensions. Our goal is to study some important homological invariants of the powers of $F$ given the knowledge of $I$ and $J$.

For a finitely generated graded $R$-module $M$, we denote by $\reg M, \depth M$, and $\lind_R M$ the Castelnuovo-Mumford regularity, depth, and linearity defect of $M$, respectively (see Section \ref{sect_background}). Our main contributions are explicit formulas or sharp bounds for $\reg F^s,\depth F^s$, and $\lind_T F^s$ via the data of $I$ and $J$. 
\begin{thm}[Corollary \ref{cor_reg_equigenerated}]
\label{thm_main1}
Assume that either $\chara k=0$ or $I$ and $J$ are monomial ideals. Moreover, assume that each of $I$ and $J$ is generated by forms of the same degree. Then for all $s\ge 1$, there is an equality 
$$
\reg F^s=\max_{i\in [1,s]} \bigl\{\reg I^i+s-i,\reg J^i+s-i \bigr\}.
$$
\end{thm}
We also have an explicit formula for $\reg F^s$ even if $I$ and $J$ are not generated by forms of the same degree (Theorem \ref{thm_reg_fiberprod_polynomialrings}). The general formula is slightly more complicated; after all, the conclusion of Theorem \ref{thm_main1} no longer holds true in full generality.

Lescot's work \cite{Les} on fiber products of local rings implies the formula
$$
\depth(T/F)=\min\{1,\depth (R/I), \depth(S/J)\}.
$$ 
The study of $\depth (T/F^s)$ provides the following (at least to us) surprising result.
\begin{thm}[Theorem \ref{thm_depth_fiberprod_polynomialrings}]
\label{thm_main2}
Assume that either $\chara k=0$ or $I$ and $J$ are monomial ideals. Moreover, assume that either $I$ or $J$ is non-zero. Then for all $s\ge 2$, there is an equality  $\depth (T/F^s)=0$.
\end{thm}
We recall that the linearity defect measures $\lind_R M$ how far $M$ is from having a linear free resolution: $M$ is componentwise linear in the sense of Herzog and Hibi \cite{HH} if and only if $\lind_R M=0$. Our third main result is
\begin{thm}[Theorems {\ref{thm_fpKoszul}} and {\ref{thm_ld_fiberprod_polynomialrings}}]
\label{thm_main3}
We always have $\lind_T F=\max\{\lind_R I,\lind_S J\}$.

Assume further that either $\chara k=0$ or $I$ and $J$ are monomial ideals. Then for all $s\ge 2$, there are inequalities 
\[
\max\{\lind_R I^s,\lind_S J^s\} \le \lind_T F^s \le \max_{i\in [1,s]} \bigl\{\lind_R (\mm^{s-i}I^i),\lind_S (\nn^{s-i}J^i)\bigr\}.
\]
\end{thm}
The first assertion of Theorem \ref{thm_main3} generalizes a result of Conca and R\"omer \cite[Theorem 4.1]{CR} which says that $\lind_T F=0$ if and only if $\lind_R I=\lind_S J=0$. We anticipate that the inequality on the right in Theorem \ref{thm_main3} is always an equality, namely for all $s\ge 2$, $\lind_T F^s = \max_{i\in [1,s]} \bigl\{\lind_R (\mm^{s-i}I^i),\lind_S (\nn^{s-i}J^i)\bigr\}$.

We are tempted to conjecture that our main results remain valid in {\it positive characteristic}. 

Similarly to \cite{NgV}, our main tool for proving Theorems \ref{thm_main1}--\ref{thm_main3} is the Tor-vanishing of certain injective maps. Let $\phi: M\to N$ be a map of finitely generated graded $R$-modules. We say that $\phi$ is {\it Tor-vanishing} if $\Tor^R_i(k,\phi)=0$ for all $i$. If $\phi$ is injective, the Tor-vanishing of $\phi$ implies strong relationship between various invariants of $M,N$ and $\Coker \phi$. For that reason, the explicit formulas for the depth and regularity of $(I+J)^s$ in \cite{NgV} follow from 
\begin{lem}[Ahangari Maleki {\cite[Proof of Theorem 2.5]{A}}, Nguyen--Vu {\cite[Theorem 4.5]{NgV}}]
\label{lem_Torvanishing_sum}
Let $(R,\mm)$ be a polynomial ring over $k$, and $I\subseteq \mm$ a homogeneous ideal. Assume that either $\chara k=0$ or $I$ is a monomial ideal. Then for all $s\ge 1$, the map $I^s\to I^{s-1}$ is Tor-vanishing.
\end{lem}
For the same reason, Theorems \ref{thm_main1}--\ref{thm_main3} in the present paper follow from
\begin{lem}[Lemma \ref{lem_Torvanishing}]
\label{lem_Torvanishing_fiberprod}
Let $(R,\mm)$ be a polynomial ring over $k$, and $I\subseteq \mm^2$ a homogeneous ideal. Assume that either $\chara k=0$ or $I$ is a monomial ideal. Then for all $1\le t\le s$, the map $\mm^{s-t}I^t\to \mm^{s-t+1}I^{t-1}$ is Tor-vanishing.
\end{lem} 
Speaking more precisely, the importance of Lemma \ref{lem_Torvanishing_sum} to \cite{NgV} is that it gives rise to certain decomposition, the so-called {\it Betti splitting} of $(I+J)^s$. (Betti splittings are discussed in Section \ref{sect_background_Bettisplit}.) In the same manner, Lemma \ref{lem_Torvanishing_fiberprod} gives rise to certain Betti splitting of $F^s$, from which we deduce Theorems \ref{thm_main1}--\ref{thm_main3}.

It is well-known that for all sufficiently large $s$, $\reg I^s$ is a linear function \cite{CHT, Kod}. From Theorem \ref{thm_main1}, we can determine the asymptotic value of $\reg F^s$ as follows. 
\begin{cor}[Corollary \ref{cor_asymptotic_reg}]
\label{cor_main_asymptotic_reg}
Assume that either $\chara k=0$ or $I$ and $J$ are monomial ideals. Assume that, moreover, each of $I$ and $J$ is generated in a single degree. Then for all $s\gg 0$, there is an equality $\reg F^s=\max\{\reg I^s,\reg J^s\}$.
\end{cor}
We present an example showing that the conclusion of Corollary \ref{cor_main_asymptotic_reg} is not true for arbitrary $I$ and $J$ (Remark \ref{rem_counterexample_reg}). Denote the first point where $\reg I^s$ becomes a linear function by $\rstab(I)$. If all the minimal homogeneous generators of $I$ and $J$ have the same degree, we can present a sharp upper bound for $\rstab(F)$. Somewhat unexpectedly, $\rstab(F)$ can be arbitrarily larger than $\max\{\rstab(I),\rstab(J)\}$ (Remark \ref{rem_rstab}).

Let us mention two applications of our main results. Both require no assumption on the characteristic. 
\begin{cor}[Corollary \ref{cor_Koszul_powers}]
\label{cor_main_asymptotically_Koszul}
For every $s\ge 1$, the following statements are equivalent:
\begin{enumerate}[\quad \rm(i)]
\item $F^i$ is componentwise linear for all $1\le i\le s$;
\item $I^i$ and $J^i$ are componentwise linear for all $1\le i\le s$.
\end{enumerate}
In particulars, all the powers of $F$ are componentwise linear if and only if all the powers of $I$ and of $J$ are so.
\end{cor}
Let $G=(V,E)$ be a finite simple graph with the vertex set $V=\{1,2,\ldots,n\}$ (where $n\ge 1$) and the edge set $E$. The {\it edge ideal} of $G$ is the following monomial ideal in $k[x_1,\ldots,x_n]$: $I(G)=(x_ix_j: \{i,j\}\in E)$. In \cite[Question 7.12]{BBH}, it was asked whether for any non-trivial graph $G$, the function $\reg I(G)^s$ is strictly increasing. As a second application, we can answer this question positively in a special case.
\begin{thm}[Corollary \ref{cor_edgeideals}]
\label{thm_main4}
Let $G=(V,E)$ be a graph. Assume that $V$ is a disjoint union of two non-empty subsets $V_1,V_2$ such that $G$ contains the complete bipartite graph with the bipartition $V_1\cup V_2$ as a subgraph. Then the function $\reg I(G)^s$ is strictly increasing.
\end{thm}
Note that the regularity of powers of a monomial ideal generated in a single degree need not be weakly increasing (see Remark \ref{rem_rstab}).

The paper is organized as follows. Section \ref{sect_background} is devoted to some prerequisites. In Section \ref{sect_fibprod}, we determine the regularity and linearity defect of fiber products. We also give a short derivation of the depth of fiber products. The regularity, depth, and linearity defect of higher powers of a fiber product are studied in Sections \ref{sect_regularity}, \ref{sect_depth}, and \ref{sect_lind}, respectively. The common method rests on a result on the Betti splitting of $F^s$ in Section \ref{sect_Bettisplit}. An application to the regularity of powers of edge ideals is presented in Section \ref{sect_edgeideals}. We prove the correctness of the example in Remark \ref{rem_counterexample_reg} mentioned above in Appendix \ref{appendix}. 

\section{Background}
\label{sect_background}

We will employ standard terminology and facts of commutative algebra, as in the books of Bruns and Herzog \cite{BH} and Eisenbud \cite{Eis}. For standard knowledge of homological commutative algebra, we refer to the monographs of Avramov \cite{Avr} and Peeva \cite{P}.

In the following, we let $(R,\mm)$ be a noetherian local ring with the unique maximal ideal $\mm$ and the residue field $k=R/\mm$, or a standard graded algebra over a field $k$, with the graded maximal ideal $\mm$. In the second situation, it is understood that $R$ is an $\N$-graded ring with $R_0=k$, and $R$ is generated as a $k$-algebra by finitely many elements of degree $1$.

\subsection{Linearity defect}
We recall the notion of linearity defect, introduced by Herzog and Iyengar \cite{HIy}, who took motivation from work of Eisenbud, Fl\o ystad and Schreyer \cite{EFS}.

Let $(R,\mm)$ be a noetherian local ring with the residue field $k=R/\mm$, and $M$ a finitely generated $R$-module. Let $X$ be the minimal free resolution of $M$:
\[
X: \cdots \xlongrightarrow{\partial} X_i\xlongrightarrow{\partial} X_{i-1} \xlongrightarrow{\partial} \cdots \xlongrightarrow{\partial} X_1 \xlongrightarrow{\partial} X_0 \longrightarrow 0.
\]
The {\it projective dimension} of $M$ over $R$ is
\[
\pdim_R M=\sup\{i: X_i\neq 0\} =\sup\{i: \Tor^R_i(k,M)\neq 0\}.
\]
The minimality of $X$ induces for each $i\ge 0$ a subcomplex
\[
\Fcc^iX: \cdots \to X_{i+1} \to X_i \to \mm X_{i-1} \to \cdots \to \mm^{i-1}X_1 \to \mm^iX_0\to 0.
\]
Associate with the filtration
\[
X=\Fcc^0X\supseteq \Fcc^1X \supseteq \Fcc^2X \supseteq \cdots 
\]
the complex 
\[
\linp^R X= \bigoplus_{i=0}^\infty \frac{\Fcc^iX}{\Fcc^{i+1}X}.
\]
We call $\linp^R X$ the {\it linear part} of $X$. Let
\[
\gr_\mm(M)=\bigoplus_{i\ge 0} \frac{\mm^iM}{\mm^{i+1}M}
\]
be the associated graded module of $M$ with respect to the $\mm$-adic filtration. The linear part $\linp^R X$ is a complex of $\gr_\mm(R)$-modules, with 
$$
(\linp^R X)_i=(\gr_\mm(X_i))(-i) \quad \text{for all $i$}.
$$ 
The {\it linearity defect} of $M$ as an $R$-module is
\[
\lind_R M=\sup\{i: H_i(\linp^R X)\neq 0\}.
\]
By convention, $\lind_R (0)=0$. In the graded situation, we define the linear part of a minimal graded free resolution, and the linearity defect of a graded module similarly. The linearity defect can be characterized by the vanishing of certain maps of Tor \cite[Theorem 2.2]{Se}. This characterization gives rise to the following "depth lemma" for the linearity defect.
\begin{lem}[Nguyen, {\cite[Proposition 2.5]{Ng}}]
\label{lem_exactseq}
Let $(R,\mm)$ be a noetherian local ring with the residue field $k$. Consider an exact sequence $0\longrightarrow M \longrightarrow P \longrightarrow N\longrightarrow 0$ of finitely generated $R$-modules. Denote 
\begin{align*}
d_P&=\inf\{m\ge 0: \Tor^R_i(k,M) \longrightarrow \Tor^R_i(k,P) ~\text{is the zero map for all $i\ge m$}\},\\
d_N &=\inf\{m\ge 0: \Tor^R_i(k,P) \longrightarrow \Tor^R_i(k,N) ~\text{is the zero map for all $i\ge m$}\},\\
d_M&=\inf\{m\ge 0: \Tor^R_{i+1}(k,N) \longrightarrow \Tor^R_i(k,M) ~\text{is the zero map for all $i\ge m$}\}.
\end{align*}
Then there are inequalities
\begin{align*}
\lind_R P &\le \max\{\lind_R M, \lind_R N, \min\{d_M, d_N\}\},\\
\lind_R N &\le \max\{\lind_R M+1, \lind_R P, \min\{d_M+1,d_P\}\},\\
\lind_R M &\le \max\{\lind_R P, \lind_R N-1, \min\{d_P,d_N-1\}\}.
\end{align*}
\end{lem}
For more information on the linearity defect and related problems, see, for example, \cite{CINR, HIy, IyR, Se}.

\subsection{Koszul modules}
\label{subsect_Koszul}
Let $(R,\mm)$ be a graded $k$-algebra, $M$  a finitely generated graded $R$-module. The {\it regularity of $M$ over $R$} is the following number
\[
\reg_R M=\sup\{j-i:\Tor^R_i(k,M)_j\neq 0\}.
\] 
If $R$ is a polynomial ring, $\reg_R M$ is the same as the Castelnuovo-Mumford regularity $\reg M$ of $M$. However, they differ in general.

We call $R$ a {\it Koszul algebra} if $\reg_R (R/\mm)=0$. We say that $M$ is a {\it Koszul module} if $\lind_R M=0$. (In previous work of \c{S}ega \cite{Se1}, Koszul modules are called modules with linear resolution.) 

In this paper, we say that $M$ has a {\it linear resolution} if there exists an integer $d$ such that $M$ is generated in degree $d$, and $\reg_R M=d$. In that case, we say $M$ has a $d$-linear resolution. Following Herzog and Hibi \cite{HH}, we call $M$ {\it componentwise linear} if for each $d$, the submodule $(M_d)$ has a $d$-linear resolution. If $R$ is a Koszul algebra, then $M$ is a Koszul module if and only if it is componentwise linear (see \cite[Theorem 3.2.8]{Ro} and \cite[Proposition 4.9]{Yan}). 

\subsection{Tor-vanishing morphisms and Betti splittings}
\label{sect_background_Bettisplit}
Let $(R,\mm)$ be a noetherian local ring with the residue field $k$. Let $\phi: M\to P$ be a morphism of finitely generated $R$-modules. We say that $\phi$ is {\it Tor-vanishing} if for all $i\ge 0$, it holds that $\Tor^R_i(k,\phi)=0$.

Let $P,I,J\neq (0)$ be proper ideals of $R$ such that $P=I+J$. Following \cite{FHV}, the decomposition of $P$ as $I+J$ is called a {\it Betti splitting} if for all $i\ge 0$, the following equality of Betti numbers holds: $\beta_i(P)=\beta_i(I)+\beta_i(J)+\beta_{i-1}(I\cap J)$. Here, as usual, $\beta_i(M)=\dim_k \Tor^R_i(k,M)$.

Betti splittings are closely related to the Tor-vanishing of certain maps.
\begin{lem}[Francisco, H\`a, and Van Tuyl {\cite[Proposition 2.1]{FHV}}]
\label{lem_criterion_Bettisplit}
The following are equivalent:
\begin{enumerate}[\quad \rm(i)]
\item The decomposition $P=I+J$ is a Betti splitting;
\item The morphisms $I\cap J \to I$ and $I\cap J \to J$ are Tor-vanishing;
\item The mapping cone construction for the map $I\cap J \to I\oplus J$ yields a minimal free resolution of $P$.
\end{enumerate}
In particular, if $P=I+J$ is a Betti splitting, then for all $i$, the natural maps $\Tor^R_i(k,I)\to \Tor^R_i(k,P)$ and $\Tor^R_i(k,J)\to \Tor^R_i(k,P)$ are injective.
\end{lem}

Let $R=k[x_1,\ldots,x_n]$ be a standard graded polynomial ring, $\mm=R_+$ and $I\subseteq \mm$ a homogeneous ideal of $R$. Denote by $\partial(I)$ the ideal generated by $\partial f/\partial x_i$, where $f$ runs through the minimal homogeneous generators of $I$ and $1\le i\le n$. The following differential criteria for detecting Tor-vanishing homomorphisms will be useful.
\begin{lem}[Ahangari Maleki {\cite[Theorem 2.5 and its proof]{A}}]
\label{lem_differential_criterion}
Assume that $k$ has characteristic zero. Let $I_1$ and $I_2$ be homogeneous ideals of $R$ such that $\partial(I_1)\subseteq I_2$. Then $I_1\subseteq \mm I_2$ and the map $I_1\to I_2$ is Tor-vanishing.
\end{lem}
If $f$ is a monomial of $R$, then the {\it support of $f$}, denoted $\supp f$, refers to the set of variables dividing $f$. If $I$ is a monomial ideal, the {\it support of $I$}, denoted $\supp I$, is the union of the supports of its minimal monomial generators. For a monomial ideal $I$, we denote by $\partial^*(I)$ the ideal generated by the monomials $f/x_i$, where $f\in I$ and $x_i\in \supp f$. By Lemma 4.2 and Proposition 4.4 in \cite{NgV}, we have
\begin{lem}[Nguyen--Vu]
\label{lem_differential_criterion_monomial}
Let $I_1$ and $I_2$ be monomial ideals of $R$ such that $\partial^*(I_1)\subseteq I_2$. Then $I_1\subseteq \mm I_2$ and the map $I_1\to I_2$ is Tor-vanishing.
\end{lem}
Another source of Tor-vanishing maps is given by
\begin{lem}
\label{lem_maxideal_Koszul}
Let $(R,\mm)$ be a Koszul local ring and $M$ a finitely generated $R$-module which is Koszul. Then the map $\mm M\to M$ is Tor-vanishing and $\mm M$ is also Koszul.
\end{lem}
\begin{proof}
This follows from the proof of \cite[Corollary 3.8]{Ng}.
\end{proof}

\subsection{Elementary facts}

\begin{lem}
\label{lem_intersect}
Let $R,S$ be affine $k$-algebras. Let $I, J$ be ideals of $R,S$, respectively. Then in $T=R\otimes_k S$, there is an equality
$I\cap J=IJ$.
\end{lem}
This standard fact follows from \cite[Lemma 2.2(i)]{NgV}.

\begin{lem}
\label{lem_tensor}
Let $R, S$ be standard graded $k$-algebras, and $M, N$ be finitely generated graded modules over $R, S$, respectively. Then for $T=R \otimes_k S$, there are equalities
\begin{align*}
\depth (M\otimes_k N)&=\depth M+\depth N,\\
\pdim_T (M\otimes_k N)&=\pdim_R M+\pdim_S N,\\
\reg_T (M\otimes_k N)&=\reg_R M+ \reg_S N,\\
\lind_T (M\otimes_k N)&=\lind_R M+\lind_S N.
\end{align*}
\end{lem}
This lemma is folklore. The first part follows from the description of local cohomology of tensor product due to Goto and Watanabe \cite[Theorem 2.2.5]{GW}. For the remaining assertions, see \cite[Lemma 2.3]{NgV}. 

Recall that a morphism of noetherian local rings $(R,\mm)\xrightarrow{\theta} (S,\nn)$ is an {\it algebra retract} if there exists a local homomorphism $ (S,\nn)\xrightarrow{\phi} (R,\mm)$ such that $\phi\circ \theta$ is the identity map of $R$. In that case, $\phi$ is called the retraction map of $\theta$.
\begin{lem}[{\cite[Lemma 2.4]{NgV}}]
\label{lem_retract}
Let $ (R,\mm)\xrightarrow{\theta} (S,\nn)$ be an algebra retract of noetherian local rings with the retraction map $ S\xrightarrow{\phi} R$. Let $I\subseteq \mm$ be an ideal of $R$. Let $J\subseteq \nn$ be an ideal containing $\theta(I)S$ such that $\phi(J)R=I$. Then there are inequalities
\begin{align*}
\lind_R (R/I) &\le \lind_S (S/J),\\
\lind_R I &\le \lind_S J. 
\end{align*}
\end{lem}

\section{Fiber products}
\label{sect_fibprod}
Let $(A,\mm_A)$ and $(B,\mm_B)$ be noetherian local rings with the same residue field $k=A/\mm_A \cong B/\mm_B$. Let $\vepsilon_A: A\to A/\mm_A$ and $\vepsilon_B: B\to B/\mm_B$ be the canonical maps. The {\it fiber product} $A\times_k B$ of $A$ and $B$ over $k$ fits into the following pullback diagram.
\begin{displaymath}
\xymatrix{A\times_k B \ar[r]^{\pi_A} \ar[d]_{\pi_B} & A \ar[d]^{\vepsilon_A}\\
           B  \ar[r]_{\vepsilon_B}  &  k
}
\end{displaymath}
Hence $A\times_k B=\{(x,y)\in A\times B: \vepsilon_A(x)=\vepsilon_B(y)\}$. Denote $P=A\times_k B$. From \cite[Lemma 1.2]{AAF}, $P$ is a local ring with the maximal ideal $\mm_A\oplus \mm_B$. We can check that $\pi_A,\pi_B$ are surjective and $\Ker \pi_A=\mm_B, \Ker \pi_B=\mm_A$. In particular, they induce ring isomorphisms $P/\mm_B \cong A$ and $P/\mm_A \cong B$.

We have an exact sequence of $P$-modules
\begin{equation}
\label{eq_exactseq_fiberprod}
0\longrightarrow A\times_k B \longrightarrow A\oplus B \xrightarrow{(\vepsilon_A,-\vepsilon_B)}  k \longrightarrow 0.
\end{equation}
The following result due to Lescot is well-known; see for example \cite[(3.2.1)]{CSV}. We include another elementary proof for it.
\begin{lem}[Lescot]
\label{lem_depth_fp}
There is an equality
\[
\depth(A\times_k B)=\min\{1,\depth A, \depth B\}.
\]
\end{lem}
\begin{proof}
If $\min\{\depth A, \depth B\}\ge 1$, using the depth lemma for the sequence \eqref{eq_exactseq_fiberprod}, we get $\depth(A\times_k B)=1$.

Assume that $\min\{\depth A, \depth B\}=0$, for example $\depth A=0$. Then there exists $x\in \mm_A$ such that $\ann_A(x)=\mm_A$. Since $(x,0)$ belongs to the maximal ideal of $P$, and $\ann_P((x,0)) \supseteq \mm_A\oplus \mm_B$, we deduce $\depth P=0$.
\end{proof}
Now we turn to the case of graded algebras. Let $A=k[x_1,\ldots,x_m]/I$ and $B=k[y_1,\ldots,y_n]/J$ be standard graded $k$-algebras, where $I\subseteq (x_1,\ldots,x_m)^2$, $J\subseteq (y_1,\ldots,y_n)^2$ are homogeneous ideals. The homomorphism $k\xrightarrow{\iota} A\times_k B$ given by $u\mapsto (u,u)$ makes $A\times_k B$ into a $k$-algebra. 

Denote $\mm=(x_1,\ldots,x_m)$, $\nn=(y_1,\ldots,y_n)$ and 
$$
P=k[x_1,\ldots,x_m,y_1,\ldots,y_n]/(I+J+\mm\nn).
$$
Write $\overline{a}$ for the residue class of $a$ in a suitable quotient ring. Any element $ u\in P$ can be written uniquely as $u=\overline{c+f(x)+g(y)}$, where $f\in \mm, g\in \nn$. The homomorphism from $P$ to $A\times_k B$ sending $u$ to $(\overline{c+f(x)},\overline{c+g(y)})$ is an isomorphism. Hence we have
\begin{lem}
\label{lem_fiberprod_graded}
There is an isomorphism of $k$-algebras $P\cong A\times_k B$.
\end{lem}

In the remaining  of this section, we determine the linearity defect and regularity of fiber products of ideals over polynomial rings, or more generally, Koszul algebras. Notably, the result does not depend on the characteristic of the base field. 

\begin{prop}
\label{prop_depthreg_fiberprod}
Let $(R,\mm)$ and $(S,\nn)$ be standard graded algebras over $k$. Let $I\subseteq \mm^2, J\subseteq \nn^2$ be homogeneous ideals. Denote $T=R\otimes_k S$ and $F=I+J+\mm\nn \subseteq T$ the fiber product of $I$ and $J$. 
\begin{enumerate}[\quad \rm(i)]
\item There is an equality $\depth(T/F)=\min\{1,\depth(R/I),\depth(S/J)\}$.

If moreover, $(R,\mm)$ and $(S,\nn)$ are standard graded polynomial rings of positive dimension, and either $I$ or $J$ is non-zero, then
\[
\depth F=\min\{2,\depth I, \depth J\}.
\]
\item Assume additionally that $R$ and $S$ are Koszul algebras. Then there is an equality $\reg_T F=\max\{\reg_R I, \reg_S J\}$.
\end{enumerate}
\end{prop}
\begin{proof}
(i) This is essentially a translation of Lemma \ref{lem_depth_fp} to the graded case. 

(ii) From \eqref{eq_exactseq_fiberprod}, we have an exact sequence
\begin{equation}
\label{eq_exactseq}
0\to \frac{T}{F} \to \frac{R}{I} \oplus \frac{S}{J} \to k \to 0.
\end{equation}

By the assumption and Lemma \ref{lem_tensor}, $\reg_T (T/(\mm+\nn))=0$. Since $\reg_R(R/I)$, $\reg_S(S/J)\ge 1$, from the sequence \eqref{eq_exactseq}, we get
\[
\reg_T (T/F)=\max\{\reg_T (R/I),\reg_T (S/J)\}.
\]
Since $0=\reg_S (S/\nn)=\reg_T (T/\nn T)=\reg_T R$, by \cite[Proposition 3.3]{CDR}, $\reg_T(R/I)=\reg_R(R/I)$. Similarly, $\reg_T(S/J)=\reg_S(S/J)$. So $\reg_T (T/F)=\max\{\reg_R (R/I)$, $\reg_S (S/J)\}$, consequently
\[
\reg_T F=\max\{\reg_R I,\reg_S J\}.
\]
The proof is completed.
\end{proof}

Recall from Section \ref{subsect_Koszul} that over a Koszul algebra, componentwise linear modules and Koszul modules are the same objects. Hence Theorem \ref{thm_fpKoszul} below is a generalization of Conca and R\"omer's \cite[Theorem 4.1]{CR}. Observe that the condition $R$ and $S$ are Koszul in \ref{thm_fpKoszul} is irredundant: if one of these rings is not Koszul, then the fiber product of $I=(0)$ and  $J=(0)$ is not Koszul, since by Lemma \ref{lem_tensor}, 
$\lind_T (\mm\nn)=\lind_R \mm +\lind_S \nn \ge 1$.

\begin{thm}
\label{thm_fpKoszul}
Let $(R,\mm), (S,\nn)$ be Koszul algebras over $k$. Let $I\subseteq \mm^2, J\subseteq \nn^2$ be homogeneous ideals of $R,S$, respectively. Denote $F=I+J+\mm\nn$ the fiber product of $I$ and $J$ in $T=R\otimes_k S$. Then there is an equality
\[
\lind_T F =\max\{\lind_R I, \lind_S J\}.
\]
\end{thm}
\begin{proof}
By the graded analog of Lemma \ref{lem_retract}, we have $\max\{\lind_R I, \lind_S J\} \le \lind_T F$. It remains to show the reverse inequality.

\noindent
\textsf{Step 1}: Denote $H=I+\mm\nn$. We claim that $F=H+J$ is a Betti splitting. First we note that $H\cap J=\mm J$.

Indeed, $\mm J \subseteq \mm \nn \subseteq H$, hence $\mm J \subseteq H\cap J$. Conversely, as $I\subseteq \mm$,
\[
H\cap J \subseteq \mm \cap J=\mm J,
\]
thanks to Lemma \ref{lem_intersect}.

Now by Lemma \ref{lem_criterion_Bettisplit}, we need to show that the map $\mm J\to H$ and $\mm J\to J$ are Tor-vanishing. Clearly the map $\mm \to R$ is Tor-vanishing, so applying $-\otimes_k J$, we get that $\mm J\to J$ is Tor-vanishing.

Since $\nn$ is Koszul and $J\subseteq \nn^2$, by Lemma \ref{lem_maxideal_Koszul}, the map $J\to \nn$ is Tor-vanishing. Applying $-\otimes_k \mm$, we get that $\mm J\to \mm \nn$ is Tor-vanishing. Since  $\mm\nn\subseteq H$, the map $\mm J\to H$ is Tor-vanishing as well. Thus $F=H+J$ is a Betti splitting.

\medskip
\noindent
\textsf{Step 2}: There is a short exact sequence
\[
0 \longrightarrow  H \longrightarrow  F \longrightarrow \frac{J}{H \cap J} =\frac{J}{\mm J} \longrightarrow 0.
\]
Since $F=H+J$ is a Betti splitting, by Lemma \ref{lem_criterion_Bettisplit}, $\Tor^T_i(k,H) \to \Tor^T_i(k,F)$ is injective for all $i\ge 0$. In particular, the connecting map $\Tor^T_i(i, J/\mm J) \to \Tor^T_{i-1}(k,H)$ is zero for all $i$. Applying Lemma \ref{lem_exactseq}, we get
\[
\lind_T F \le \max\biggl\{\lind_T H, \lind_T \frac{J}{\mm J}\biggr\}=\max\{\lind_T H, \lind_S J\}.
\]
The equality follows from Lemma \ref{lem_tensor} and the fact that $R$ is Koszul.

\medskip
\noindent
\textsf{Step 3}: Arguing similarly as above, we have a Betti splitting $H=I+\mm\nn$ and 
\[
\lind_T H \le \max\{\lind_T (\mm\nn),\lind_R I\}=\lind_R I.
\]
The equality follows from Lemma \ref{lem_tensor} and the fact that $R$ and $S$ are Koszul.

From Steps 2 and 3, we get $\lind_T H \le \max\{\lind_R I, \lind_S J\}$, as desired.
\end{proof}

\section{Betti splittings for powers of fiber products}
\label{sect_Bettisplit} 
The following lemma is the crux of our results on higher powers of fiber products.
\begin{lem}
\label{lem_Torvanishing}
Let $(R,\mm)$ be a standard graded polynomial ring over $k$, and $I\subseteq \mm^2$ a homogeneous ideal. Let $s\ge 1$ be an integer. Assume further that one of the following conditions hold:
\begin{enumerate}[\quad \rm(i)]
\item  $\chara k=0$;
\item $I$ is a monomial ideal;
\item $I^t$ is Koszul for all $1\le t\le s$.
\end{enumerate}
Then for every $1\le t\le s$, the map $\mm^{s-t}I^t \to \mm^{s-t+1}I^{t-1}$ is Tor-vanishing.
\end{lem}
\begin{proof}
First assume that $\chara k=0$. We apply the criterion of Lemma \ref{lem_differential_criterion}. Note that
\[
\partial(\mm^{s-t}I^t)\subseteq \partial(\mm^{s-t})I^t+\mm^{s-t}\partial(I^t) \subseteq \mm^{s-t-1}I^t+ \mm^{s-t}\partial(I)I^{t-1},
\]
which is contained in $\mm^{s-t+1}I^{t-1}$, since $I\subseteq \mm^2$.

If $I$ is a monomial ideal, we argue as above, replacing Lemma \ref{lem_differential_criterion} by Lemma \ref{lem_differential_criterion_monomial}.

Finally, assume that $I^t$ is Koszul for all $1\le t\le s$. By the graded analogue of Lemma \ref{lem_maxideal_Koszul}, $\mm^{s-t+1}I^{t-1}$ is Koszul. Since $I\subseteq \mm^2$, $\mm^{s-t}I^t\subseteq \mm(\mm^{s-t+1}I^{t-1})$. Hence again by Lemma \ref{lem_maxideal_Koszul}, the map $\mm^{s-t}I^t \to \mm^{s-t+1}I^{t-1}$ is Tor-vanishing.
\end{proof}
We know of no instance where the conclusion of Lemma \ref{lem_Torvanishing} fails in positive characteristic.
\begin{quest}
Let $(R,\mm)$ be a standard graded polynomial ring over $k$, and $I\subseteq \mm^2$ a homogeneous ideal. Is it true that for all $1\le t\le s$, the map $\mm^{s-t}I^t \to \mm^{s-t+1}I^{t-1}$ is Tor-vanishing?
\end{quest}
It is convenient for the subsequent discussions to introduce
\begin{notn}
\label{notn_RandS}
Let $(R,\mm), (S,\nn)$ be standard graded polynomial rings over $k$ such that their Krull dimensions are positive. Let $I\subseteq \mm^2$ and $J\subseteq \nn^2$ be homogeneous ideals of $R$ and $S$. Denote $T=R\otimes_k S$ and $F=I+J+\mm\nn$ the fiber product of $I$ and $J$. Denote $H=I+\mm\nn$.
\end{notn}
The important consequence of Lemma \ref{lem_Torvanishing} is
\begin{prop}
\label{prop_Betti-splitting_power}
Employ Notation \ref{notn_RandS}. Let $s\ge 1$ be an integer.
\begin{enumerate}[\quad \rm (1)]
 \item  There is an equality $F^s=H^s+\sum_{i=1}^s (\mm\nn)^{s-i}J^i$.
\item For each $1\le t\le s$, denote $G_t=H^s+\sum_{i=1}^t (\mm\nn)^{s-i}J^i$. Denote $G_0=H^s$. Then for every $1\le t\le s$, there is an equality
\[
G_{t-1}\cap (\mm\nn)^{s-t}J^t =\mm^{s-t+1}\nn^{s-t}J^t.
\]
\item Assume further that one of the following conditions holds:
\begin{enumerate}[\quad \rm(i)]
\item  $\chara k=0$;
\item $J$ is a monomial ideal;
\item $J^t$ is Koszul for all $1\le t\le s$.
\end{enumerate}
Then for all $1\le t\le s$, the decomposition $G_t=G_{t-1}+(\mm\nn)^{s-t}J^t$ is a Betti splitting.
\end{enumerate}
\end{prop}
\begin{proof}
(1) We use induction on $s\ge 1$. For $s=1$, the equality $F=H+J$ is obvious. Assume that the equality holds for $s\ge 1$, we prove it for $s+1$.

We have $F^{s+1}=H^{s+1}+JF^s$. Using the induction hypothesis,
\[
JF^s=JH^s+\sum_{i=1}^s (\mm\nn)^{s-i}J^{i+1}=JH^s+\sum_{i=2}^{s+1} (\mm\nn)^{s+1-i}J^i.
\] 
Now $H^s=(\mm\nn)^s+\sum_{j=1}^{s-1}I^j(\mm\nn)^{s-j}$. Hence
\[
F^{s+1}=H^{s+1}+JF^s=H^{s+1}+\sum_{i=1}^{s+1}(\mm\nn)^{s+1-i}J^i+\sum_{j=1}^{s-1}JI^j(\mm\nn)^{s-j}.
\]
It remains to show that $\sum_{j=1}^{s-1}JI^j(\mm\nn)^{s-j} \subseteq H^{s+1}$. Indeed, since $I\subseteq \mm^2,J\subseteq \nn^2$, for any $1\le j\le s-1$,
\[
JI^j(\mm\nn)^{s-j} =(JI)I^{j-1}(\mm\nn)^{s-j} \subseteq I^{j-1}(\mm\nn)^{s+2-j} \subseteq H^{s+1}.
\]
This finishes the induction.

(2) Note that $(\mm\nn)^{s-t+1}J^{t-1} \subseteq G_{t-1}$, so $\mm^{s-t+1}\nn^{s-t}J^t\subseteq G_{t-1}\cap (\mm\nn)^{s-t}J^t$. For the reverse inclusion, using $H\subseteq \mm$, we have
\[
G_{t-1}=H^s+\sum_{i=1}^t (\mm\nn)^{s-t+1}J^{t-1} \subseteq \mm^{s-t+1}.
\]
Hence
\[
G_{t-1}\cap (\mm\nn)^{s-t}J^t \subseteq \mm^{s-t+1} \cap \nn^{s-t}J^t=\mm^{s-t+1}\nn^{s-t}J^t,
\]
where the equality holds by Lemma \ref{lem_intersect}.

(3) To prove that $G_t=G_{t-1}+(\mm\nn)^{s-t}J^t$ is a Betti splitting, by Lemma \ref{lem_criterion_Bettisplit}, we have to show that $\mm^{s-t+1}\nn^{s-t}J^t \to G_{t-1}$ and $\mm^{s-t+1}\nn^{s-t}J^t\to (\mm\nn)^{s-t}J^t$ are Tor-vanishing maps.

Since $\mm^{s-t}$ is Koszul, we have that $\mm^{s-t+1} \to \mm^{s-t}$ is Tor-vanishing. Applying $-\otimes_k (\nn^{s-t}J^t)$, we get that $\mm^{s-t+1}\nn^{s-t}J^t\to (\mm\nn)^{s-t}J^t$ is Tor-vanishing.

Since the inclusion  $\mm^{s-t+1}\nn^{s-t}J^t \to G_{t-1}$ factors through $\mm^{s-t+1}\nn^{s-t}J^t \to (\mm\nn)^{s-t+1}J^{t-1}$, it remains to show that the last map is Tor-vanishing. As the map $-\otimes_k \mm^{s-t+1}$ is flat, we reduce to showing that $\nn^{s-t}J^t \to \nn^{s-t+1}J^{t-1}$ is Tor-vanishing. This follows from Lemma \ref{lem_Torvanishing}. The proof is concluded.
\end{proof}

\section{Regularity}
\label{sect_regularity}
Now we determine the regularity of higher powers of fiber products. 
\begin{thm}
\label{thm_reg_fiberprod_polynomialrings}
Employ Notation \ref{notn_RandS}. Let $s\ge 1$ be an integer. Assume that one of the following conditions holds: 
\begin{enumerate}[\quad \rm(i)]
\item $\chara k=0$;
\item $I$ and $J$ are monomial ideals;
\item $I^t$ and $J^t$ are Koszul for all $1\le t\le s$.
\end{enumerate}
Then there is an equality
\[
\reg F^s = \max_{i\in [1,s]} \bigl\{\reg (\mm^{s-i}I^i)+s-i,\reg (\nn^{s-i}J^i)+s-i\bigr\}.
\]
\end{thm}
\begin{proof}
We apply Proposition \ref{prop_Betti-splitting_power} and employ the notation there. For each $1\le t\le s$, we have $G_t=G_{t-1}+(\mm\nn)^{s-t}J^t$ and $G_{t-1}\cap (\mm\nn)^{s-t}J^t=\mm^{s-t+1}\nn^{s-t}J^t$. Thus we get an induced exact sequence
\[
0\to G_{t-1} \to G_t \to \frac{(\mm\nn)^{s-t}J^t}{\mm^{s-t+1}\nn^{s-t}J^t} \to 0.
\] 
By Proposition \ref{prop_Betti-splitting_power} and Lemma \ref{lem_criterion_Bettisplit}, for all $i$, the map $\Tor^T_i(k,G_{t-1})\to \Tor^T_i(k,G_t)$ is injective. This yields an exact sequence
$$
0 \to \Tor^T_i(k,G_{t-1}) \to \Tor^T_i(k,G_t) \to \Tor^T_i\left(k,\frac{(\mm\nn)^{s-t}J^t}{\mm^{s-t+1}\nn^{s-t}J^t}\right) \to 0.
$$
Hence 
\begin{align*}
\reg G_t &= \max\left\{\reg G_{t-1}, \reg \frac{(\mm\nn)^{s-t}J^t}{\mm^{s-t+1}\nn^{s-t}J^t}\right\} \\
         &= \max\left\{\reg G_{t-1}, \reg \left(\frac{\mm^{s-t}}{\mm^{s-t+1}}\right)+\reg (\nn^{s-t}J^t)\right\}\\
         &= \max \bigl\{\reg G_{t-1}, \reg(\nn^{s-t}J^t)+s-t \bigr\}.
\end{align*}
The second equality follows from Lemma \ref{lem_tensor}. Now as $G_s=F^s$ and $G_0=H^s$, we conclude that
\[
\reg F^s = \max\{\reg H^s, \reg (\nn^{s-t}J^t)+s-t: 1\le t\le s\}.
\]
Now $H=I+\mm\nn$ can be viewed as the fiber product of $I$ and $(0)\subseteq S$. Arguing as above, we get
\begin{align*}
\reg H^s &= \max\{\reg (\mm\nn)^s, \reg (\mm^{s-t}I^t)+s-t: 1\le t\le s\}\\
            &=\max\{2s,\reg (\mm^{s-t}I^t)+s-t: 1\le t\le s\} \\
            & = \max\{\reg (\mm^{s-t}I^t)+s-t: 1\le t\le s\}.
\end{align*}
The last equality holds since $\reg(I^s)\ge 2s$. Putting everything together, we get 
\[
\reg F^s = \max\{\reg (\mm^{s-t}I^t)+s-t, \reg (\nn^{s-t}J^t)+s-t: 1\le t\le s\}.
\]
The proof is concluded.
\end{proof}
If each of $I$ and $J$ is generated by forms of the same degree, then the formula in Theorem \ref{thm_reg_fiberprod_polynomialrings} can be simplified.
\begin{cor}
\label{cor_reg_equigenerated}
Keep using Notation \ref{notn_RandS}. Assume that either $\chara k=0$ or $I$ and $J$ are monomial ideals. Assume further that each of $I$ and $J$ is generated by forms of the same degree. Then for all $s\ge 1$, there is an equality
\[
\reg F^s=\max_{i\in [1,s]}\bigl\{\reg I^i+s-i, \reg J^i+s-i\bigr\}.
\]
\end{cor}
To prove the corollary, first we describe the effect on the regularity of a module after multiplying it with the graded maximal ideal. Let $(R,\mm)$ be a standard graded $k$-algebra. For a finitely generated graded $R$-module $M$, denote $t_0(M)=\sup\{i: (M/\mm M)_i \neq 0\}$, the maximal degree of a minimal homogeneous generator of $M$. The following lemma essentially follows from the proof of Eisenbud--Ulrich's \cite[Proposition 1.4]{EU}. For the sake of clarity, we present the detailed argument. 
\begin{lem}[Eisenbud--Ulrich]
\label{lem_maxideal_reg}
Let $(R,\mm)$ be a standard graded $k$-algebra. Let $M\neq 0$ be a finitely generated graded $R$-modules such that $\depth M\ge 1$. 
\begin{enumerate}[\quad \rm (i)]
\item For every $i\ge 1$, there is an equality
\[
\reg (\mm^iM)=\max\biggl\{\reg M, \reg \frac{M}{\mm^iM}+1 \biggr\}.
\]
In particular, setting $i=1$, we get 
\[
\reg(\mm M) =\max\{\reg M, 1+t_0(M)\}.
\]
\item Assume further that $M$ is generated in a single degree. Then for all $i\ge 1$,
\[
\reg(\mm^i M) =\max\{\reg M, i+t_0(M)\}.
\]
In particular, let $Q$ be a standard graded polynomial ring over $k$ and $Q\to R$ be a surjection of graded $k$-algebras, then for all $i\ge \reg M-t_0(M)$, $\mm^iM$ has a linear resolution as a $Q$-module.
\end{enumerate}

\end{lem}
\begin{proof}
(i) Let $r\in \mm$ be an $M$-regular element. Then $r^iM\neq 0$ for all $i\ge 1$. Hence $\mm^iM\neq 0$ for all $i\ge 1$. Consider the exact sequence
\[
0\to \mm^iM\to M \to \frac{M}{\mm^iM}\to 0.
\]
Since $H^0_\mm(M)=0$, we get an exact sequence
\[
0\to H^0_\mm(M/\mm^iM)=M/\mm^iM \to H^1_\mm(\mm^iM) \to H^1_\mm(M)\to 0,
\]
and for all $\ell \ge 2$, we have $H^\ell_\mm(\mm^iM) \cong H^\ell_\mm(M)$. This implies that
\begin{align*}
\reg (\mm^iM) &\ge \max\{u+1: H^1_\mm(\mm^iM)_u\neq 0\} \ge \reg (M/\mm^iM)+1,\\
\reg (\mm^iM) &\ge \reg M.
\end{align*}
Therefore 
\[
\reg (\mm^iM)=\max\biggl\{\reg M, \reg \frac{M}{\mm^iM}+1 \biggr\}.
\]
Setting $i=1$, we get the remaining assertion.

(ii) If $M$ is generated in a single degree, then clearly 
$$
\reg (M/\mm^iM)=\max\{t: (M/\mm^iM)_t\neq 0\}=i+t_0(M)-1.
$$ Hence by part (i), for all $i\ge 1$,
\[
\reg (\mm^iM)=\max\bigl\{\reg M, i+t_0(M) \bigr\}.
\]
The last assertion follows easily.
\end{proof}
\begin{rem}
\label{rem_Sega}
By the graded analog of a result due to \c{S}ega \cite[Theorem 3.2(1)]{Se1}, we have the following interesting result, which coincides with Lemma \ref{lem_maxideal_reg}(ii) when $R$ is regular: Let $(R,\mm)$ be a Koszul algebra, and $M$ a finitely generated graded $R$-module, which is generated in a single degree. Then for all $i\ge 0$, there is an equality
\[
\reg_R (\mm^iM)=\max\{\reg_R M,i+t_0(M)\}.
\]
Note that {\it ibid.} deals with local rings and what is written as $\reg_R(M)$ there really means $\reg_{\gr_\mm R} \left(\gr_\mm(M)\right)$ in our situation. 

It is known that if $R$ is a Koszul algebra, then for all finitely generated graded $R$-module $M$, the inequality $\reg_R M\le \reg M$ holds \cite{AE}, and strict inequality may happen.
\end{rem}

\begin{rem}
For $i\ge 2$, the equality $\reg(\mm^iM) =\max\{\reg M, i+t_0(M)\}$ in Lemma \ref{lem_maxideal_reg}(ii) does not hold in general, when $M$ is not generated in a single degree. For example, let $I=(a^3,ab^2,ac^2,a^2bc)\subseteq R=k[a,b,c]$ and $i=2$. Then 
$\reg(\mm^2I)=5< \max\{\reg I, 2+t_0(I)\}=6$.
\end{rem}
\begin{proof}[Proof of Corollary \ref{cor_reg_equigenerated}]
In view of Theorem \ref{thm_reg_fiberprod_polynomialrings}, it suffices to show that
\[
\max_{i\in [1,s]}\{\reg(\mm^{s-i}I^i)+s-i\} = \max_{i\in [1,s]}\{\reg I^i+s-i\}.
\]
(The same equality holds for $J$.) Assume that $I$ is generated by forms of degree $p\ge 2$. By Lemma \ref{lem_maxideal_reg}, $\reg(\mm^{s-i}I^i)=\max\{\reg I^i,s+(p-1)i\}$. Hence
\[
\max_{i\in [1,s]}\{\reg(\mm^{s-i}I^i)+s-i\} =\max_{i\in [1,s]}\{\reg I^i+s-i,2s+(p-2)i\}.
\]
It remains to notice that for $1\le i\le s$, $2s+(p-2)i\le ps \le \reg I^s$.
\end{proof}
\begin{rem}
\label{rem_reg_non-equigenerated}
In view of Theorem \ref{thm_reg_fiberprod_polynomialrings}, Lemma \ref{lem_maxideal_reg}, we see that if either $\chara k=0$ or $I$ and $J$ are monomial ideals, then
\[
\reg F^s \ge \max_{i\in [1,s]}\{\reg I^i+s-i, \reg J^i+s-i\}.
\]
One may ask whether the equality always happens, as in Corollary \ref{cor_reg_equigenerated}. The answer is "No". Choose $R=k[a,b,c]$, $I=(a^4,a^3b,ab^3,b^4,a^2b^2c^4)$, $S=k[x], J=(x^4)$. We have $\reg I=\reg I^2=8$, $\reg (\mm I)=9$. The fiber product $F$ satisfies
\[
\reg F^2=\reg (\mm I)+1=10  > \max_{i\in [1,2]}\{\reg I^i+2-i,\reg J^i+2-i\}=9.
\]
\end{rem}
By results of Cutkosky-Herzog-Trung \cite{CHT} and Kodiyalam \cite{Kod}, there exist constants $p,q$ such that $\reg I^s=ps+q$ for all $s\gg 0$. Denote by $\rstab(I)$ the minimal positive integer $m$ such that for all $s\ge m$, $\reg I^s=ps+q$. In many interesting situations, we have a simple asymptotic formula for $\reg F^s$. 
\begin{cor}
\label{cor_asymptotic_reg}
Keep using Notation \ref{notn_RandS}. Assume that either $\chara k=0$ or $I$ and $J$ are monomial ideals. Assume further that both $I$ and $J$ satisfy one of the following conditions: 
\begin{enumerate}[\quad \rm (i)]
\item All the minimal homogeneous generators have degree $2$;
\item All the minimal homogeneous generators have degree at least $3$;
\item The subideal generated by elements of degree 2 is integrally closed, e.g. $I$ and $J$ are squarefree monomial ideals.
\end{enumerate}
Then for all $s\gg 0$, there is an equality $\reg F^s=\max\{\reg I^s,\reg J^s\}$.
\end{cor}

\begin{proof}
It suffices to show that with the extra hypothesis, we have for all $s\gg 0$ the equality
\[
\reg I^s=\max_{i\in [1,s]} \{\reg (\mm^{s-i}I^i)+s-i\}.
\]
(And the same equality holds for $J$.)

Let $p$ be the minimal number such that $I$ is contained in the integral closure of $I_{\le p}=(x\in I: \deg x\le p)$. By Kodiyalam's \cite[Theorem 5]{Kod}, there exist integral constants $d\ge 1$ and $q\ge 0$ such that $\reg I^n=pn+q$ for all $n\ge d$. By the hypothesis, either $p\ge 3$, or $p=2$ and $I$ is generated in degree $2$. Indeed, this is clear if $I$ satisfies either (i) or (ii). Assume that $I$ satisfies (iii). If $p\ge 3$, we are done.  If $p=2$, by the hypothesis $I_{\le 2}$ is integrally closed, so $I=I_{\le 2}$, namely $I$ is generated in degree $2$.

If $I$ is generated in a single degree, by Lemma \ref{lem_maxideal_reg}, there exists an integer $N$ such that for every $1\le i\le d-1$, $\mm^nI^i$ has a linear resolution for all $n\ge N$.

Take $s\ge \max\{N+d-1,\reg I^i-2i: 1\le i\le d\}$. We show that 
$$
\reg I^s=\max_{i\in [1,s]} \{\reg (\mm^{s-i}I^i)+s-i\}.
$$

For $d\le i\le s$, using successively the exact sequence
\[
0\to \mm M \to M\to M/\mm M\to 0,
\]
we get $\reg(\mm^{s-i}I^i) \le s-i+\reg I^i$. Hence 
$$
\reg(\mm^{s-i}I^i)+s-i\le 2s+\reg I^i-2i=2s+(p-2)i+q\le  ps+q=\reg I^s.
$$ 

For $1\le i\le d-1$, consider first the case $p\ge 3$. As above, we get $\reg(\mm^{s-i}I^i)+s-i\le 2s+\reg I^i-2i \le ps$ for all $1\le i\le d-1$, thanks to the fact that $p\ge 3$ and the choice of $s$. Hence we are done in this case.

Now assume that $p=2$ and $I$ is generated in degree $2$. For $1\le i\le d-1$, we have $s-i\ge s-d+1\ge N$. Hence $\mm^{s-i}I^i$ has a linear resolution. The ideal $\mm^{s-i}I^i$ is generated in degree $s+i$, so
\[
\reg(\mm^{s-i}I^i)+s-i=2s \le \reg I^s,
\] 
as desired. This concludes the proof.
\end{proof}
\begin{rem}
\label{rem_counterexample_reg}
The following example shows that the conclusion of Corollary \ref{cor_asymptotic_reg} is not true for arbitrary $I$ and $J$, even if both of them are monomial ideals.
Consider $R=\Q[a,b,c,d,x,y,z,t], \mm=R_+$ and
\[
I=(a^2,b^2,c^2,d^2,abx,cdx,acy,bdy,adz,bcz,cdyzt).
\]
Consider $S=\Q[u],\nn=(u)$ and $J=(u^2)$. It can be shown that for all $s\ge 3$,
\[
\reg I^s = 2s+3 < 2s+4 = \reg (\mm^{s-2}I^2)+s-2.
\]
Since the proof is long and technical (about seven pages), we defer the detailed argument to Theorem \ref{thm_counterexample} in Appendix \ref{appendix}.

Let $F=I+J+\mm \nn$ be the fiber product of $I$ and $J$. Then from Theorem \ref{thm_reg_fiberprod_polynomialrings}, for all $s\ge 3$,
\[
\reg F^s \ge \reg (\mm^{s-2}I^2)+s-2 =2s+4> 2s+3=\max \{\reg I^s,\reg J^s\}.
\]
\end{rem}
If moreover, $I$ and $J$ are both generated in a common degree, we can give a sharp bound for $\rstab(F)$ in terms of the data of $I$ and $J$.
\begin{cor}
\label{cor_rstab}
Keep using Notation \ref{notn_RandS}. Assume that either $\chara k=0$ or $I$ and $J$ are monomial ideals. Assume further that both $I$ and $J$ are generated in the same degree $p\ge 2$. Denote $g=\max\{\reg I^i-i,\reg J^j-j: 1\le i\le \rstab(I)-1, 1\le j\le \rstab(J)-1\}$. Denote $q_1=\reg I^r-pr$, $q_2=\reg J^r-pr$ where $r=\max\{\rstab(I),\rstab(J)\}$. Then for all $s\ge \max\{\rstab(I),\rstab(J),\lceil g/(p-1)\rceil\}$, there is an equality
\[
\reg F^s=ps+\max\{q_1,q_2\}.
\]
In particular, $\rstab(F) \le \max\{\rstab(I),\rstab(J),\lceil g/(p-1)\rceil\}$.
\end{cor}
\begin{proof}
By Corollary \ref{cor_reg_equigenerated}, 
\[
\reg F^s=\max_{i\in [1,s]}\{\reg I^i+s-i,\reg J^i+s-i\}.
\]
Denote $N=\max\{\rstab(I),\rstab(J),\lceil g/(p-1)\rceil\}$. It suffices to show that for all $s\ge N$,
\begin{equation}
\label{eq_regIs=max}
\reg I^s=\max_{i\in [1,s]}\{\reg I^i+s-i\}.
\end{equation}
(The similar equality holds for $J$.) Indeed, for $s\ge N$ and $1\le i\le \rstab(I)-1$, we have
\[
\reg I^s\ge ps \ge s+g\ge  \reg I^i+s-i.
\]
For $\rstab(I)\le i\le s$, we have $\reg I^i+s-i=(p-1)i+q_1+s \le ps+q_1=\reg I^s$. Hence \eqref{eq_regIs=max} is true.
\end{proof}

\begin{rem}
\label{rem_rstab}
The following example shows that the bound for $\rstab(F)$ in Corollary \ref{cor_rstab} is sharp, and that the difference $\rstab(F)- \max\{\rstab(I),\rstab(J)\}$ can be arbitrarily large. Let $n\ge 2$ be an integer, $R=k[a,b,x_1,\ldots,x_n]$, 
$$
I=(a^4,a^3b,ab^3,b^4)(x_1,\ldots,x_n)^2+a^2b^2(x_1^2,\ldots,x_n^2).
$$ 
Let $S=k[y]$ and $J=(y^6)$. Both $I$ and $J$ are generated in degree $6$. We show that in this case $\rstab(I)=2,\rstab(J)=1$, while the fiber product satisfies $\rstab(F)=\lceil (n+4)/5\rceil$. Indeed, we will prove the following claims:
\begin{enumerate}
\item $I^s=(a,b)^{4s}(x_1,\ldots,x_n)^{2s}$ for all $s\ge 2$.
\item $\reg I=n+5$, $\reg I^s=6s$ for all $s\ge 2$.
\item $\reg F^s=\max\{n+4+s,6s\}$ for all $s\ge 1$.
\end{enumerate}
From (2), we get $\rstab(I)=2$ and from (3), $\rstab(F)=\lceil (n+4)/5\rceil$.

For (1): Let $H=(a^4,a^3b,ab^3,b^4)$. Then $H^2=(a,b)^8$, $H^3=(a,b)^{12}$, whence $H^s=(a,b)^{4s}$ for all $s\ge 2$.

From $H(x_1,\ldots,x_n)^2 \subseteq I \subseteq (a,b)^4(x_1,\ldots,x_n)^2$, we obtain
\[
(a,b)^{4s}(x_1,\ldots,x_n)^{2s} =H^s(x_1,\ldots,x_n)^{2s} \subseteq I^s \subseteq (a,b)^{4s}(x_1,\ldots,x_n)^{2s}.
\]
This proves (1).

For (2): By part (1) and Lemma \ref{lem_tensor}, we get for all $s\ge 2$
\[
\reg I^s=\reg (a,b)^{4s}+\reg (x_1,\ldots,x_n)^{2s}=6s.
\]
To prove $\reg I=n+5$, first we observe that
\[
(a^4,a^3b,ab^3,b^4)(x_1,\ldots,x_n)^2 \cap a^2b^2(x_1^2,\ldots,x_n^2) =a^2b^2(a,b)(x_1^2,\ldots,x_n^2).
\]
The right hand side clearly is contained in the left hand side. The other containment follows from the next display
\begin{align*}
&(a^4,a^3b,ab^3,b^4)(x_1,\ldots,x_n)^2 \cap a^2b^2(x_1^2,\ldots,x_n^2)\\
 &\subseteq (a^3,b^3)\cap (a^2b^2)\cap (x_1^2,\ldots,x_n^2)\\
                                    & =a^2b^2(a,b) \cap (x_1^2,\ldots,x_n^2)=a^2b^2(a,b)(x_1^2,\ldots,x_n^2),
\end{align*}
where the last equality holds by Lemma \ref{lem_intersect}.

The last observation yields an exact sequence
\[
0\to U=(a^4,a^3b,ab^3,b^4)(x_1,\ldots,x_n)^2 \to I \to \frac{a^2b^2(x_1^2,\ldots,x_n^2)}{a^2b^2(a,b)(x_1^2,\ldots,x_n^2)}=V \to 0.
\]
Using Lemma \ref{lem_tensor},
\begin{align*}
\reg U &=\reg (a^4,a^3b,ab^3,b^4)+\reg (x_1,\ldots,x_n)^2 =7,\\
\reg V &=\reg \frac{a^2b^2}{(a,b)a^2b^2}+\reg (x_1^2,\ldots,x_n^2)=4+(n+1)=n+5.
\end{align*}
Since  $\reg U\le  \reg V$, we conclude that $\reg I=\reg V=n+5$.

For (3): Since $I$ and $J$ are monomial ideals generated in degree 6, by Corollary \ref{cor_reg_equigenerated}, for all $s\ge 1$,
\[
\reg F^s=\max_{i\in [1,s]}\{\reg I^i+s-i,\reg J^i+s-i\}.
\]
The desired conclusion follows from (2).
\end{rem}

\section{Depth}
\label{sect_depth}
Next we have the following intriguing result.
\begin{thm}
\label{thm_depth_fiberprod_polynomialrings}
Employ Notation \ref{notn_RandS}. Assume that not both $I$ and $J$ are zero. Assume further that either $\chara k=0$ or $I$ and $J$ are monomial ideals. Then for all $s\ge 2$, there is an equality $\depth F^s = 1$.
\end{thm}
\begin{proof}
It suffices to prove that $\depth F^s\le 1$ for all $s\ge 2$. For this, we show that for all $s\ge 2$, there is an equality
\[
\depth F^s = \min_{i\in [1,s]} \bigl\{2, \depth (\mm^{s-i}I^i), \depth (\nn^{s-i}J^i)\bigr\}.
\]
The proof is similar to that of Theorem \ref{thm_reg_fiberprod_polynomialrings}. We only need to observe additionally that thanks to the hypothesis $\dim R,\dim S\ge 1$ and Lemma \ref{lem_tensor}, we have $\depth(\mm^s\nn^s)=\depth \mm^s+\depth \nn^s=2$.

It is harmless to assume that $I\neq (0)$. Then $\depth (\mm^{s-i}I^i)=1$ for any $1\le i\le s-1$, by the easy Lemma \ref{lem_maximalideal_depth} below. The proof is concluded.
\end{proof}
\begin{lem}
\label{lem_maximalideal_depth}
Let $(R,\mm)$ be a noetherian local ring. Let $M\neq 0$ be a finitely generated $R$-module such that $\depth M\ge 1$. Then $\depth (\mm M)=1$.
\end{lem}
\begin{proof}
Apply the depth lemma to the exact sequence $0\to \mm M\to M \to \dfrac{M}{\mm M}\to 0$.
\end{proof}

\section{Linearity defect}
\label{sect_lind}

The main result of this section is
\begin{thm}
\label{thm_ld_fiberprod_polynomialrings}
Employ Notation \ref{notn_RandS}. Let $s\ge 1$ be an integer. Then there is an inequality 
$\max\{\lind_R I^s,\lind_S J^s\} \le \lind_T F^s$.

Assume further that one of the following conditions holds: 
\begin{enumerate}[\quad \rm(i)]
\item $\chara k=0$;
\item $I$ and $J$ are monomial ideals;
\item $I^i$ and $J^i$ are Koszul for all $1\le i\le s$.
\end{enumerate}
Then there is an upper bound for $\lind_T F^s$:
\[
\lind_T F^s \le \max_{i\in [1,s]} \bigl\{\lind_R (\mm^{s-i}I^i),\lind_S (\nn^{s-i}J^i)\bigr\}.
\]
\end{thm}

\begin{proof}
By the graded analog of Lemma \ref{lem_retract}, we get the first inequality 
$$
\max \{\lind_R I^s,\lind_S J^s\} \le \lind_T F^s.
$$

For the remaining inequality, we argue similarly as for Theorem \ref{thm_reg_fiberprod_polynomialrings}. We apply Proposition \ref{prop_Betti-splitting_power} and employ the notation there. For each $1\le t\le s$, there is an exact sequence
\[
0\to G_{t-1} \to G_t \to \frac{(\mm\nn)^{s-t}J^t}{\mm^{s-t+1}\nn^{s-t}J^t} \to 0.
\] 
By Proposition \ref{prop_Betti-splitting_power} and Lemma \ref{lem_criterion_Bettisplit}, the connecting map 
$$
\Tor^T_i\biggl(k,\frac{(\mm\nn)^{s-t}J^t}{\mm^{s-t+1}\nn^{s-t}J^t}\biggr) \to \Tor^T_{i-1}(k,G_{t-1})
$$
is zero for all $i$. Hence by Lemma \ref{lem_exactseq},
\[
\lind_T G_t \le \max\biggl\{\lind_T G_{t-1}, \lind_T \frac{(\mm\nn)^{s-t}J^t}{\mm^{s-t+1}\nn^{s-t}J^t}\biggr\}=\max \{\lind_T G_{t-1},\lind_S (\nn^{s-t}J^t)\}.
\]
The equality follows from Lemma \ref{lem_tensor}. Now as $G_s=F^s$ and $G_0=H^s$, we conclude that
\[
\lind_T F^s \le \max\{\lind_T H^s, \lind_S (\nn^{s-t}J^t): 1\le t\le s\}.
\]
Now $H=I+\mm\nn$ can be viewed as the fiber product of $I$ and $(0)\subseteq S$. Arguing as above, we get
\begin{align*}
\lind_T H^s &\le \max\{\lind_T (\mm\nn)^s, \lind_R (\mm^{s-t}I^t): 1\le t\le s\}\\
            &=\max\{\lind_R (\mm^{s-t}I^t): 1\le t\le s\}.
\end{align*}
The last equality holds since thanks to Lemma \ref{lem_tensor}, $\lind_T (\mm\nn)^s=\lind_R \mm^s+\lind_S \nn^s=0$. Putting everything together, we get 
\[
\lind_T F^s \le \max\{\lind_R (\mm^{s-t}I^t), \lind_S (\nn^{s-t}J^t): 1\le t\le s\}.
\]
The proof is concluded.
\end{proof}

\begin{cor}
\label{cor_Koszul_powers}
Employ Notation \ref{notn_RandS}. For any given integer $s\ge 1$, the following statements are equivalent:
\begin{enumerate}[\quad \rm(i)]
\item For all $1\le i\le s$, $I^i$ and $J^i$ are Koszul ideals.
\item For all $1\le i\le s$, $F^i$ is a Koszul ideal.
\end{enumerate}
\end{cor}
\begin{proof}
(ii) $\Longrightarrow$ (i): Use the first inequality in Theorem \ref{thm_ld_fiberprod_polynomialrings}.

(i) $\Longrightarrow$ (ii): We prove that $\lind_T F^i=0$ for $i=s$. The same argument works for the smaller $i$s. By Theorem \ref{thm_ld_fiberprod_polynomialrings}, 
\[
\lind_T F^s \le \max \{\lind_R (\mm^{s-i}I^i), \lind_S (\nn^{s-i}J^i): 1\le i\le s\}.
\]
Now apply Lemma \ref{lem_maxideal_Koszul}, $\lind_R (\mm^{s-i}I^i)=\lind_S (\nn^{s-i}J^i)$ for every $1\le i\le s$. Hence $\lind_T F^s=0$.
\end{proof}
Although we can only provide an upper bound for $\lind_T F^s$, we know of no example where the strict inequality occurs. Hence we ask
\begin{quest}
\label{quest_ld}
Employ Notation \ref{notn_RandS}. Is it true that for every $s\ge 1$, the equality
\[
\lind_T F^s = \max \{\lind_R (\mm^{s-i}I^i), \lind_S (\nn^{s-i}J^i): 1\le i\le s\}
\]
holds?
\end{quest}
Similarly, we wonder if the formulas for $\depth F^s$ and $\reg F^s$ in Theorems \ref{thm_reg_fiberprod_polynomialrings} and \ref{thm_depth_fiberprod_polynomialrings} always hold regardless of the characteristic of $k$.

\section{An application to edge ideals}
\label{sect_edgeideals}
An immediate consequence of Theorem \ref{thm_reg_fiberprod_polynomialrings} is that the regularity of powers of a fiber product is always increasing, at least in characteristic $0$.
\begin{cor}
\label{cor_increasingreg_fiberprod}
Let $(R,\mm)$ and $(S,\nn)$ be standard graded polynomial rings of positive dimensions over $k$. Let $I\subseteq \mm^2, J\subseteq \nn^2$ be homogeneous ideals, and $F$ their fiber product in $R\otimes_k S$. Assume that either $\chara k=0$ or $I$ and $J$ are monomial ideals. Then for all $s\ge 1$, there is an inequality $\reg F^s <\reg F^{s+1}$.
\end{cor}
\begin{proof}
Applying Lemma \ref{lem_maxideal_reg} for $M=\mm^{s-i}I^i$, we get for all $1\le i\le s$ that $\reg(\mm^{s+1-i}I^i)\ge \reg(\mm^{s-i}I^i)$. In particular, $\reg(\mm^{s+1-i}I^i)+s+1-i > \reg (\mm^{s-i}I^i)+s-i$. Taking the maximum over all $1\le i\le s$ and using Theorem \ref{thm_reg_fiberprod_polynomialrings}, we get $\reg F^{s+1}>\reg F^s$.
\end{proof}
Note that the regularity of powers of a monomial ideal need not be increasing, even if it is generated in a single degree. A counterexample is the ideal $I$ in Remark \ref{rem_rstab} for $n\ge 8$.

A (finite, simple) graph $G$ is a pair $(V,E)$, where $V$ is a finite set, and $E$ is a set consisting of two-element subsets of $V$. We call $V$ the set of vertices and $E$ the set of edges of $G$. We say that $G$ is {\it bipartite} if $V$ can be written as a disjoint union $V_1\cup V_2$ such that for every $i\in \{1,2\}$, no edge of $E$ connects two vertices in $V_i$. In that case, we call $V_1\cup V_2$ a bipartition of $G$. If additionally $E=V_1\times V_2$ then we say that $G$ is a {\it complete bipartite} graph. A graph $G'=(V',E')$ is said to be a subgraph of $G$ if $V\subseteq V'$ and $E\subseteq E'$.

If $G=(V,E)$ is a graph with the vertex set $V=\{1,\ldots,n\}$, where $n\ge 1$, the edge ideal of $G$ is 
\[
I(G)=(x_ix_j: \{i,j\}\in E)\subseteq k[x_1,\ldots,x_n].
\]
It is known that the function $\reg I(G)^s$ is weakly increasing in the following cases:
\begin{enumerate}
\item The complement graph of $G$ is chordal \cite[Theorem 3.2]{HHZ};
\item $G$ is $2K_2$-free and cricket free \cite[Theorems 3.4 and 6.17]{Ba}; 
\item Certain classes of bipartite graph including unmixed, weakly chordal and $P_6$-free bipartite graphs \cite[Corollary 5.1]{JNS};
\item $G$ contains a leaf \cite[Theorem 5.1]{CHHKTT}. This subsumes the class of whiskered graphs treated in \cite[Corollary 5.1(2)]{JNS}.
\end{enumerate}
For related work on some more special cases, we refer to \cite{AlBa,ABS,BHT,CRJNP,Er1,Er2,JS,MBMVV,MSY}.

We can provide a class of graphs $G$ for which the function $\reg I(G)^s$ is strictly increasing.
\begin{cor}
\label{cor_edgeideals}
Let $G=(V,E)$ be a graph. Assume that $V$ can be written as a disjoint union of two non-empty subsets $V_1$ and $V_2$ such that $G$ contains the complete bipartitie graph with the bipartition $V_1\cup V_2$ as a subgraph. Then the function $\reg I(G)^s$ is strictly increasing.
\end{cor}
\begin{proof}
For $i=1,2$, let $E_i$ be the subset of $E$ consisting of edges with both vertices lying in $V_i$, and $G_i=(V_i,E_i)$. Denote $I_i=I(G_i)$. Then the assumption implies that $I(G)$ is the fiber product of $I_1$ and $I_2$. Since $I_1$ and $I_2$ are monomial ideals, the result follows from Corollary \ref{cor_increasingreg_fiberprod}.
\end{proof}

\appendix

\section{A special monomial ideal}
\label{appendix}

Our goal in this section is to prove 
\begin{thm}
\label{thm_counterexample}
Let $k=\Q$, $R=k[a,b,c,d,x,y,z,t]$, $\mm=R_+$, and 
\[
I=(a^2,b^2,c^2,d^2,abx,cdx,acy,bdy,adz,bcz,cdyzt).
\]
Then $\reg I^n = 2n+3$ for all $n\ge 3$ and $\reg (\mm^sI^2) = s+8$ for all $s\ge 0$. In particular, for all $n\ge 3$,
\[
\reg I^n = 2n+3 < 2n+4 = \reg (\mm^{n-2}I^2)+n-2.
\]
\end{thm}
The proof is not difficult, but technically involved.

First, we have the following standard lemma.
\begin{lem}
\label{lem_reg_linearforms}
Let $M$ be a monomial ideal in a polynomial ring $R$ over $k$, and $x_1,\ldots,x_n$ pairwise distinct variables of $R$. Then there is an inequality
\begin{align*}
\reg M  \le  \mathop{\max_{0\le s\le n}}_{1\le i_1<\cdots<i_s\le n}\left\{\reg \biggl((M,x_{i_1},\ldots,x_{i_s}):\prod_{j\in [n] \setminus \{i_1,\ldots,i_s\}}x_j\biggr)+n-s\right\}.
\end{align*}
\end{lem}
\begin{proof}
We use induction on $n\ge 1$. For $n=1$, we have to show
\begin{equation}
\label{eq_ineq_reg_oneform}
\reg M \le \max \{\reg (M,x_1),\reg (M:x_1)+1\},
\end{equation}
which is immediate from the exact sequence
$$
0\to \frac{R}{M:x_1}(-1) \to \frac{R}{M} \to \frac{R}{(M,x_1)} \to 0.
$$
For $n\ge 2$, we only need to use \eqref{eq_ineq_reg_oneform} and the induction hypothesis for each of the ideals $(M,x_1), M:x_1$ and the variables $x_2,\ldots,x_n$.  It is useful to note that since $M$ is a monomial ideal, for all $2\le i_1<\cdots <i_s\le n$,
\[
(M:x_1)+(x_{i_1},\ldots,x_{i_s})=(M,x_{i_1},\ldots,x_{i_s}):x_1.
\]
The desired inequality then follows.
\end{proof}
We also need some elementary computations. 
\begin{notn}
\label{notn_ideals}
Let $\qq=(a,b,c,d)$, $H=(a^2,b^2,c^2,d^2)$. Furthermore, denote
\begin{align*}
K &=(a^2,b^2,c^2,d^2,abx,cdx),\\ 
L &=(a^2,b^2,c^2,d^2,ab,cd),\\
V_1 &=H^2+H(a,b)(c,d)+ab(c^2,d^2)+(a^2,b^2)cd,\\
V_2 &=H^2+H(a,c)(b,d)+ac(b^2,d^2)+(a^2,c^2)bd.
\end{align*}
For each $s\ge 0$, denote $W_s =H^{s+1}abcd(y,z)+tcd(c^2,d^2)^{s+2}$.
\end{notn}

\begin{lem}
\label{lem_identities}
With Notation \ref{notn_ideals}, the following hold:
\begin{enumerate}[\quad \rm(i)]
\item $I^4=HI^3$.
 \item $I^3:xyz=\qq^6$.
 \item $K^3:x^2=L^3$.
 \item $K^3:x+(x)=H^2L+(x)$.
 \item For each $s\ge 1$, $\qq^{2s+1}\subseteq L^s$. 
\end{enumerate}
\end{lem}
\begin{proof}
The equalities (i)--(iv) can be checked by inspection or with Macaulay2 \cite{GS}.

For (v): this follows from the more general observation: For two arbitrary ideals $J_1,J_2$, and all $s\ge 1$, we have
\[
(J_1+J_2)^{2s+1} \subseteq (J_1^2+J_2^2)^s.
\]
In our case $J_1=(a,b), J_2=(c,d)$.

To prove the observation, we show for all $a,b\ge 0$ with $a+b=2s+1$ that $J_1^aJ_2^b \subseteq (J_1^2+J_2^2)^s$. Without loss of generality, we can assume that $a$ is odd and $b$ is even. So $a=2i+1, b=2j$, where $i,j\ge 0$ and $i+j=s$. Now the conclusion follows from the fact that
\[
J_1^aJ_2^b=J_1^{2i+1}J_2^{2j}\subseteq (J_1^2)^i(J_2^2)^j \subseteq (J_1^2+J_2^2)^s. 
\]
The proof is completed.
\end{proof}

\begin{lem}
With Notation \ref{notn_ideals}, the following hold:
\begin{enumerate}[\quad \rm(i)]
\label{lem_identities_part2}
\item  $\qq^5\subseteq V_1$.
 \item $(I^3,x):(yz)=(x)+HV_1+W_0$.
\item $(I^3,y):(xz)=(y)+HV_2+Habcd(x,z)$.
\item  For all $s\ge 0$, 
$$
(H^sI^3,x):(yz)=(x)+H^{s+1}V_1+W_s.
$$
\item For all $s\ge 0$, 
\begin{align*}
H^{s+1}V_1 \cap H^{s+1}abcd(y,z)=\qq H^{s+1}abcd(y,z).
\end{align*}
\item For all $s\ge 0$, 
\begin{align*}
H^{s+1}V_1 \cap W_s=\qq W_s.
\end{align*}
\item \textup{(}Recall that $W_s =H^{s+1}abcd(y,z)+tcd(c^2,d^2)^{s+2}$.\textup{)} For all $s\ge 0$,
\begin{align*}
H^{s+1}abcd(y,z)  \cap tcd(c^2,d^2)^{s+2} &=abcd(y,z)t(c^2,d^2)^{s+2} \subseteq \qq W_s,\\
H^{s+1}abcd(y,z) \cap \qq W_s &= \qq H^{s+1}abcd(y,z),\\
tcd(c^2,d^2)^{s+2} \cap \qq W_s &= \qq tcd(c^2,d^2)^{s+2}.
\end{align*}
\end{enumerate}
\end{lem}
\begin{proof}
The relations (i)--(iii) can be checked by inspection or with Macaulay2 \cite{GS}.

(iv) Since $\supp H \cap \{x,y,z\}=\emptyset$, we get
\begin{align*}
(H^sI^3,x):(yz)&=H^s((I^3,x):(yz))+(x)=(x)+H^{s+1}V_1+H^sW_0\\
               &=(x)+H^{s+1}V_1+H^{s+1}abcd(y,z)+tcd(c^2,d^2)^2H^s.
\end{align*}
The second equality follows from part (ii).

Since $(c^2,d^2)\subseteq H$, 
$$
W_s=H^{s+1}abcd(y,z)+tcd(c^2,d^2)^{s+2} \subseteq H^{s+1}abcd(y,z)+tcd(c^2,d^2)^2H^s.
$$
Hence
\[
(x)+H^{s+1}V_1+W_s \subseteq (H^sI^3,x):(yz).
\]
To prove the reverse inclusion, we only need to show that
\begin{equation}
\label{eq_inclusion_1}
tcd(c^2,d^2)^2H^s \subseteq H^{s+1}V_1+tcd(c^2,d^2)^{s+2}.
\end{equation}
We have $H^s=(c^2,d^2)^s+(a^2,b^2)H^{s-1}$, hence
\[
tcd(c^2,d^2)^2H^s=tcd(c^2,d^2)^{s+2}+tcd(c^2,d^2)^2(a^2,b^2)H^{s-1}.
\]
Since $cd(a^2,b^2)\subseteq V_1$ and $(c^2,d^2)\subseteq H$, we get
\[
tcd(c^2,d^2)^2(a^2,b^2)H^{s-1}=tcd(a^2,b^2)(c^2,d^2)^2H^{s-1} \subseteq V_1H^{s+1},
\]
which yields \eqref{eq_inclusion_1}.

(v) First we prove the inclusion 
\[
H^{s+1}V_1 \cap H^{s+1}abcd(y,z) \supseteq \qq H^{s+1}abcd(y,z).
\]
This follows from the chain 
$$
\qq abcd\subseteq (a^2,b^2)cd+ab(c^2,d^2) \subseteq V_1.
$$
For the reverse inclusion, take a monomial $f\in H^{s+1}V_1 \cap H^{s+1}abcd(y,z)$, and assume that $f\notin \qq H^{s+1}abcd(y,z)$. We can write $f=abcduv$, where $u\in \{y,z\}$ and $v\in H^{s+1}$. Since $f\notin \qq H^{s+1}abcd(y,z)$, $v\notin \qq H^{s+1}$. Hence $v=v_1v_2$, where $\supp v_2\cap \{a,b,c,d\}=\emptyset$ and $v_1$ is a minimal generator of $H^{s+1}$. Without loss of generality, let $u=y$. Now
\[
f=abcdyv_1v_2=(yv_2)abcdv_1 \in H^{s+1}V_1,
\]
so $abcdv_1 \in H^{s+1}V_1$. On the other hand, any minimal monomial generator of $H^{s+1}V_1$ is divisible by a polynomial of $a^2,b^2,c^2,d^2$ of degree $s+2$. Hence $abcdv_1\notin H^{s+1}V_1$. This contradiction finishes the proof of (v).

(vi) Since $W_s=H^{s+1}abcd(y,z)+tcd(c^2,d^2)^{s+2}$, from (v), it suffices to show that
\[
H^{s+1}V_1\cap tcd(c^2,d^2)^{s+2}=\qq tcd(c^2,d^2)^{s+2}.
\]
First we show that the containment "$\supseteq$". Note that
\[
\qq cd(c^2,d^2)=(c^2,d^2)cd(a,b)+(c^2,d^2)cd(c,d) \subseteq H(a,b)(c,d)+(c^2,d^2)^2 \subseteq V_1.
\]
As $(c^2,d^2)^{s+1}\subseteq H^{s+1}$, we get the desired containment.

For the reverse containment, take a monomial $f\in H^{s+1}V_1\cap tcd(c^2,d^2)^{s+2}$, and assume that $f\notin \qq tcd(c^2,d^2)^{s+2}$. Lemma \ref{lem_intersect} yields
\[
(a,b) \cap tcd(c^2,d^2)^{s+2}=(a,b)tcd(c^2,d^2)^{s+2} \subseteq \qq tcd(c^2,d^2)^{s+2},
\]
hence $f\notin (a,b)$. Note that $H\subseteq (a,b)+(c^2,d^2), V_1\subseteq (a,b)+(c^2,d^2)^2$, so
\[
H^{s+1}V_1\subseteq (a,b)+(c^2,d^2)^{s+3}.
\]
This implies $f\in (c^2,d^2)^{s+3}$. Together with the fact that $f\in tcd(c^2,d^2)^{s+2}$, we see that 
$$
f\in tcd(c^2,d^2)^{s+2}(c,d)\subseteq \qq tcd(c^2,d^2)^{s+2}.
$$
This contradicts the assumption $f\notin \qq tcd(c^2,d^2)^{s+2}$. Hence (vi) is established.

(vii) For the first assertion, it suffices to prove the equality. Applying Lemma \ref{lem_intersect}, we have
\begin{align*}
H^{s+1}abcd(y,z)  \cap tcd(c^2,d^2)^{s+2} &\subseteq ab(y,z) \cap tcd(c^2,d^2)^{s+2}\\
                  &=abcd(y,z)t(c^2,d^2)^{s+2}.
\end{align*}
Clearly
\begin{equation}
\label{eq_inclusion_2}
abcd(y,z)t(c^2,d^2)^{s+2} \subseteq \qq H^{s+1}abcd(y,z) \cap \qq tcd(c^2,d^2)^{s+2},
\end{equation}
hence the reverse containment is true.

Next we prove the second assertion
\[
H^{s+1}abcd(y,z) \cap \qq W_s = \qq H^{s+1}abcd(y,z).
\]
Since $W_s=H^{s+1}abcd(y,z)+tcd(c^2,d^2)^{s+2}$, this reduces to
\[
H^{s+1}abcd(y,z) \cap \qq tcd(c^2,d^2)^{s+2} \subseteq \qq H^{s+1}abcd(y,z).
\]
This follows from the first assertion and the containment \eqref{eq_inclusion_2}. The third assertion is proved similarly.
\end{proof}

\begin{lem}
\label{lem_regularity_powers}
Keep using Notation \ref{notn_ideals}. Then for all $s\ge 1$ and all $i\ge 0$, the following hold:
\begin{enumerate}[\quad \rm (i)]
\item $\reg (H^s\qq^i)=\max\{2s+3,2s+i\}$.
\item $\reg H^sL^{i+1} \le 2s+2i+3$.
\item $\reg H^sV_1 \le 2s+5$.
\end{enumerate}
\end{lem}
\begin{proof}
All the relevant ideals are supported in $S=k[a,b,c,d]$, hence to prove the assertions it suffices to work in $S$.

(i) For $i=0$, the equality $\reg H^s=2s+3$ holds, for example, by inspecting the Eagon-Northcott resolution. Using Lemma \ref{lem_maxideal_reg}(ii), we get
\[
\reg H^s\qq^i =\max\{\reg H^s,2s+i\}=\max\{2s+3,2s+i\}.
\]

(ii) Since $L\subseteq \qq^2$, there is an exact sequence
\[
0\to H^sL^{i+1} \to H^s\qq^{2i+2} \to \frac{H^s\qq^{2i+2}}{H^sL^{i+1}} \to 0.
\]
By Lemma \ref{lem_identities}(v), $\qq^{2i+3}\subseteq L^{i+1}$, hence $H^s\qq^{2i+2}/(H^sL^{i+1})$ is killed by $\qq$. Hence it has regularity $2s+2i+2$. From part (i),
\[
\reg (H^s\qq^{2i+2})=\max\{2s+3,2s+2i+2\}\le 2s+2i+3.
\]
Therefore from the last exact sequence, $\reg H^sL^{i+1} \le 2s+2i+3$. 

(iii) Consider the exact sequence
\[
0\to H^sV_1 \to H^s\qq^4 \to \frac{H^s\qq^4}{H^sV_1} \to 0.
\]
By Lemma \ref{lem_identities_part2}(i), $H^s\qq^4/(H^sV_1)$ is killed by $\qq$. Hence it has regularity $2s+4$. From part (i), $\reg H^s\qq^4=2s+4$, so $\reg H^sV_1\le 2s+5$, as desired.
\end{proof}

\begin{proof}[Proof of Theorem \ref{thm_counterexample}]
We proceed through several steps.

\medskip
\noindent
\textsf{Step 1:}  First we start with the easier statement that $\reg (\mm^sI^2)\ge s+8$ for all $s\ge 0$. In fact, it suffices to show that $f=abcdxyzt^{s+1} \in \mm^sI^2\setminus (\mm^{s+1}I^2)$, since $\deg f=s+8$.

Clearly $f=t^s(abx)(cdyzt)\in \mm^sI^2$. Assume that $f\in \mm^{s+1}I^2$ for some $s\ge 0$. There are two minimal generators $u_1,u_2\in I$ such that $f=gu_1u_2$ where $g\in \mm^{s+1}$. As each of $a^2,b^2,c^2,d^2$ does not divide $f$,  both $u_1$ and $u_2$ belong to $ \{abx,cdx,acy,bdy,adz,bcz,cdyzt\}$. We claim that either $u_1$ or $u_2$ must be $cdyzt$.

Indeed, assume the contrary. Then either both $u_1$ and $u_2$ are divisible by one of the variables $a,b,c,d$, or both of them are divisible by one of the variables $x,y,z$. In both case, we get a contradiction, as $f\notin (a^2,b^2,c^2,d^2,x^2,y^2,z^2)$. Thus we can assume $u_1=cdyzt$. This forces $u_2\notin (c,d)$, so the only option is $u_2=abx$. But then $g=f/(u_1u_2)=t^s\notin \mm^{s+1}$. This contradiction shows that $f\notin \mm^{s+1}I^2$.

Now we show by induction on $s\ge 0$ that $\reg (\mm^sI^2)\le s+8$. The case $s=0$ can be checked directly with Macaulay2 \cite{GS}. For $s\ge 1$, we have
\[
\reg(\mm^sI^2) \le s+\reg I^2=s+8.
\]
Therefore $\reg (\mm^sI^2)=s+8$ for all $s$.

\medskip
\noindent
\textsf{Step 2:} We show that $\reg I^n\ge 2n+3$ for all $n\ge 3$. It suffices to show that $f_1=c^{2n-1}dyzt\in I^n\setminus (\mm I^n)$, since $\deg f_1=2n+3$.

Clearly $f_1=(c^2)^{(n-1)}(cdyzt)\in I^n$. Assume that $f_1\in \mm I^n$ for some $n\ge 3$. Since $f_1\notin (a,b,d^2,x)$, we must have
\[
f_1 \in \mm (c^2,cdyzt)^n.
\]
We can write $f_1=vu_1\cdots u_n$, where $v\in \mm, u_i \in \{c^2,cdyzt\}$. One of the $c_i$s must be $cdyzt$, otherwise $f_1\in (c^{2n})$. Assume that $u_n=cdyzt$. Then $vu_1\cdots u_{n-1}=c^{2n-2}$. Since $\deg u_i\ge 2$ and $\deg v\ge 1$, we get
\[
2n-2=\deg(c^{2n-2})\ge 1+2(n-1)=2n-1,
\]
a contradiction. Hence $f_1\notin \mm I^n$.

\medskip
\noindent
\textsf{Step 3:} The remaining space is devoted to the more tricky statement that $\reg I^n\le 2n+3$ for all $n\ge 3$. We will keep using Notation \ref{notn_ideals}. Thanks to Lemma \ref{lem_identities}, $I^{s+3}=H^sI^3$ for all $s\ge 0$, so it suffices to show that $\reg H^sI^3 \le 2s+9$ for all such $s$.

Using Lemma \ref{lem_reg_linearforms} for $M=H^sI^3$ and the variables $x,y,z$, we get
\begin{align*}
&\reg M \le \\
& \max \biggl\{\reg (M,x,y,z), \reg (M,y,z):x+1,\reg (M,x,z):y+1, \reg (M,x,y):z+1,\\
& \qquad \qquad \reg (M,x):(yz)+2, \reg (M,y):(xz)+2, \reg (M,z):(xy)+2,\\
& \qquad \qquad \reg (M:xyz)+3 \biggr\}.
\end{align*}
We show that for every $s\ge 0$, each term in the right-hand side is at most $2s+9$.

\medskip
\noindent
\textsf{Step 4:} Since $x,y,z\notin \supp H$, using  Lemma \ref{lem_identities}(ii), we get
\[
M:xyz=H^s(I^3:xyz)=H^s\qq^6.
\]
By Lemma \ref{lem_regularity_powers}(i), if $s\ge 1$, $\reg (M:xyz)+3=2s+9$. If $s=0$ clearly the same thing is true.

\medskip
\noindent
\textsf{Step 5:} Clearly $(M,x,y,z)=(x,y,z)+H^{s+3}$, hence
\[
\reg (M,x,y,z)=\reg H^{s+3}=2(s+3)+3=2s+9.
\]

\medskip
\noindent
\textsf{Step 6:} We show that $\reg (M,y,z):x +1\le 2s+9$.

We have $(M,y,z):x=(y,z)+H^s(K^3:x)$. As $y$ and $z$ do not belong to the support of $H^s(K^3:x)$, we deduce $\reg (M,y,z):x=\reg H^s(K^3:x)$. It remains to show that 
\begin{equation}
\label{eq_regHsK3:x}
\reg H^s(K^3:x)\le 2s+8.
\end{equation}
 By Lemma \ref{lem_identities}(iii)--(iv), we have
\begin{align*}
K^3:x^2 &=L^3,\\
K^3:x+(x) &=H^2L+(x). 
\end{align*}
This yields the third equality in the following chain
\begin{align*}
\reg H^s(K^3:x) &=\reg (H^sK^3:x) \\
                & \le \max \biggl\{\reg (H^sK^3:x^2)+1,\reg (H^sK^3:x+(x))\biggr\}\\
                &=\max \biggl\{\reg H^s(K^3:x^2)+1, \reg \bigl(H^s(K^3:x+(x))+(x)\bigr)\biggr\}\\
                &=\max \biggl\{\reg H^sL^3+1, \reg \bigl(H^{s+2}L+(x)\bigr)\biggr\}\\
                &=\max \biggl\{\reg H^sL^3+1, \reg H^{s+2}L \biggr\}.
\end{align*}
Now \eqref{eq_regHsK3:x} follows since by Lemma \ref{lem_regularity_powers}(ii)
\[
\max\{\reg H^sL^3, \reg H^{s+2}L\} \le 2s+7.
\]

\medskip
\noindent
\textsf{Step 7:} Denote $K_1=(a^2,b^2,c^2,d^2,acy,bdy)$ and $K_2=(a^2,b^2,c^2,d^2,adz,bcz)$. Arguing as in Step 6, we have
\begin{align*}
(M,x,z):y &=(x,z)+H^s(K_1^3:y),\\
(M,x,y):z &=(x,y)+H^s(K_2^3:z).
\end{align*}
Following the same step, we deduce
\[
\max \{\reg (M,x,z):y+1, \reg (M,x,y):z+1\} \le 2s+9.
\]

\medskip
\noindent
\textsf{Step 8:} Next we show that 
\begin{equation}
\label{eq_ineq_reg(M,x):yz}
\reg (M,x):(yz)+2\le 2s+9.
\end{equation}
From Lemma \ref{lem_identities_part2}(iv),
\[
(M,x):(yz)=(x)+H^{s+1}V_1+W,
\]
where for simplicity $W=H^{s+1}abcd(y,z)+tcd(c^2,d^2)^{s+2}$. As $x$ does not belong to the support of $H^{s+1}V_1+W$, $\reg (M,x):(yz)=\reg (H^{s+1}V_1+W)$. The equality 
\[
H^{s+1}V_1\cap W=\qq W
\]
from Lemma \ref{lem_identities_part2}(vi) yields an exact sequence
\[
0\to H^{s+1}V_1 \to H^{s+1}V_1+W \to \frac{W}{\qq W}\to 0.
\]
So we obtain
\begin{equation}
\label{eq_ineq_bound_reg_(M,x):yz}
\reg (M,x):(yz)=\reg (H^{s+1}V_1+W) \le \max \biggl\{\reg H^{s+1}V_1, \reg \frac{W}{\qq W}\biggr\}.
\end{equation}
By Lemma \ref{lem_identities_part2}(vii) we get the first equality in the following display
\begin{align*}
\frac{W}{\qq W} &= \frac{H^{s+1}abcd(y,z)}{\qq H^{s+1}abcd(y,z)} \bigoplus \frac{tcd(c^2,d^2)^{s+2}}{\qq tcd(c^2,d^2)^{s+2}}\\
                & \cong \left(\frac{H^{s+1}abcd}{\qq H^{s+1}abcd}\otimes_k (y,z)\right) \bigoplus \frac{cd(c^2,d^2)^{s+2}}{\qq cd(c^2,d^2)^{s+2}}(-1).
\end{align*}
Working in $S$, we get
\[
\reg \frac{H^{s+1}abcd}{\qq H^{s+1}abcd}=\reg \frac{cd(c^2,d^2)^{s+2}}{\qq cd(c^2,d^2)^{s+2}}=2s+6,
\]
hence together with Lemma \ref{lem_tensor}, $\reg (W/\qq W)=2s+7$. Thanks to \eqref{eq_ineq_bound_reg_(M,x):yz}, in order to get \eqref{eq_ineq_reg(M,x):yz}, we are left with the inequality
\[
\reg H^{s+1}V_1\le 2s+7.
\]
The latter follows from Lemma \ref{lem_regularity_powers}(iii).

\medskip
\noindent
\textsf{Step 9:} Next we show that 
\begin{equation}
\label{eq_ineq_reg_(M,y):xz}
\reg (M,y):(xz)+2\le 2s+9.
\end{equation}

By Lemma \ref{lem_identities_part2}(iii), 
\[
(M,y):(xz)=H^s((I^3,y):(xz))+(y)=H^{s+1}V_2+H^{s+1}abcd(x,z)+(y).
\]
As $y$ does not belong to the support of $H^{s+1}V_2+H^{s+1}abcd(x,z)$,
\[
\reg (M,y):(xz)=\reg (H^{s+1}V_2+H^{s+1}abcd(x,z)).
\]
By symmetry, we deduce from Lemma \ref{lem_identities_part2}(v) the equality  
\[
H^{s+1}V_2\cap H^{s+1}abcd(x,z)=\qq H^{s+1}abcd(x,z).
\]
The latter yields an exact sequence
\[
0\to H^{s+1}V_2 \to H^{s+1}V_2+H^{s+1}abcd(x,z) \to \frac{H^{s+1}abcd(x,z)}{\qq H^{s+1}abcd(x,z)}\to 0.
\]
The last term in the sequence is isomorphic to 
\[
\frac{H^{s+1}abcd}{\qq H^{s+1}abcd}\otimes_k (x,z).
\]
Clearly
\[
\reg \frac{H^{s+1}abcd}{\qq H^{s+1}abcd}=2s+6,
\]
so by Lemma \ref{lem_tensor},
\[
\reg \frac{H^{s+1}abcd}{\qq H^{s+1}abcd}\otimes_k (x,z)=2s+7.
\]
Therefore from the last exact sequence, we reduce \eqref{eq_ineq_reg_(M,y):xz} to showing that 
$$
\reg H^{s+1}V_2 \le 2s+7.
$$
This follows by symmetry from Lemma \ref{lem_regularity_powers}(iii).

\medskip
\noindent
\textsf{Step 10:} By symmetry, arguing as in Step 9, we get $\reg (M,z):(xy)+2\le 2s+9$.

Putting everything together, we get $\reg M\le 2s+9$, as desired.
\end{proof}


{\small 
\begin{finrem*}
Some discussions in this paper, for instance a slightly more general form of Corollary \ref{cor_Koszul_powers}, appeared in Section 5 of the unpublished preprint \cite{NgV1a}. Nevertheless, we have better proofs and much more interesting materials to present here, thanks to the additional input brought in by Section \ref{sect_Bettisplit}. 

An abridged version of this paper \cite{NgV3} has appeared online, and will be published in the Acta Mathematica Vietnamica. Note that the discussion between lines -1 to -6 concerning certain result of \c{S}ega before Lemma 5.3 and Lemma 5.3(ii) in {\it loc. cit.} are not quite accurate, and should be corrected as in Remark \ref{rem_Sega} and Lemma \ref{lem_maxideal_reg}(ii) (respectively), of this paper. The remaining results of \cite{NgV3} are not affected by this small error.
\end{finrem*} 
}

\section*{Acknowledgments}
The first named author is grateful to the hospitality and support of the Department of Mathematics, Otto von Guericke Universit\"at Magdeburg, where large parts of this work were finished. The second named author is partially supported by the Vietnam National Foundation for Science and Technology Development under grant number 101.04-2016.21.

\end{document}